\documentclass[twoside,11pt]{article}

\usepackage[abbrvbib, preprint]{jmlr2e}

\usepackage{mathtools}
\usepackage{xfrac,enumerate}
\usepackage{graphicx}

\usepackage{tikz}
\usetikzlibrary{patterns}
\usetikzlibrary{decorations.pathreplacing}
\usepackage[outline]{contour}
\contourlength{2pt}

\newcommand{\R}{\mathbb{R}}
\newcommand{\bigo}{\mathcal{O}}

\newcommand{\sign}{{\rm sign}}

\DeclarePairedDelimiterX{\norm}[1]{\lVert}{\rVert}{#1}
\DeclarePairedDelimiterX{\abs}[1]{\lvert}{\rvert}{#1}
\DeclarePairedDelimiterX{\0norm}[1]{\lVert}{\rVert_{0}}{#1}
\DeclarePairedDelimiterX{\1norm}[1]{\lVert}{\rVert_{1}}{#1}
\DeclarePairedDelimiterX{\2norm}[1]{\lVert}{\rVert_{2}}{#1}
\DeclarePairedDelimiterX{\nnorm}[1]{\lVert}{\rVert_{n}}{#1}
\DeclarePairedDelimiterX{\2nnorm}[1]{\lVert}{\rVert_{n}^2}{#1}
\DeclareMathOperator*{\argmin}{arg\,min}

\jmlrheading{1}{2020}{1-30}{4/00}{10/00}{meila00a}{Francesco Ortelli and Sara van de Geer}

\ShortHeadings{Adaptive Rates for Total Variation Image Denoising}{Ortelli and van de Geer}
\firstpageno{1}

\usepackage{Sweave}
\begin{document}
\Sconcordance{concordance:TV2D-JMLR-reviewed.tex:TV2D-JMLR-reviewed.Rnw:%
1 87 1 1 0 32 1}
\Sconcordance{concordance:TV2D-JMLR-reviewed.tex:./TV2D-JMLR-Section1.Rnw:ofs 121:%
1 72 1}
\Sconcordance{concordance:TV2D-JMLR-reviewed.tex:./TV2D-JMLR-Section2.Rnw:ofs 194:%
1 72 1}
\Sconcordance{concordance:TV2D-JMLR-reviewed.tex:./TV2D-JMLR-Section3.Rnw:ofs 267:%
1 56 1 1 20 1 2 56 1}
\Sconcordance{concordance:TV2D-JMLR-reviewed.tex:./TV2D-JMLR-Section4.Rnw:ofs 382:%
1 57 1}
\Sconcordance{concordance:TV2D-JMLR-reviewed.tex:./TV2D-JMLR-Section5.Rnw:ofs 440:%
1 89 1}
\Sconcordance{concordance:TV2D-JMLR-reviewed.tex:./TV2D-JMLR-Section6.Rnw:ofs 530:%
1 172 1}
\Sconcordance{concordance:TV2D-JMLR-reviewed.tex:./TV2D-JMLR-Section7.Rnw:ofs 703:%
1 440 1}
\Sconcordance{concordance:TV2D-JMLR-reviewed.tex:./TV2D-JMLR-Section8.Rnw:ofs 1144:%
1 160 1}
\Sconcordance{concordance:TV2D-JMLR-reviewed.tex:./TV2D-JMLR-Section9.Rnw:ofs 1305:%
1 8 1}
\Sconcordance{concordance:TV2D-JMLR-reviewed.tex:TV2D-JMLR-reviewed.Rnw:ofs 1314:%
130 13 1}
\Sconcordance{concordance:TV2D-JMLR-reviewed.tex:./TV2D-JMLR-Section10.Rnw:ofs 1328:%
1 161 1}
\Sconcordance{concordance:TV2D-JMLR-reviewed.tex:TV2D-JMLR-reviewed.Rnw:ofs 1490:%
145 15 1}

\title{Adaptive Rates for Total Variation Image Denoising}

\author{\name Francesco Ortelli \email fortelli@ethz.ch \\
       \addr Seminar f\"{u}r Statistik, ETH Z\"{u}rich \\
       R\"{a}mistrasse 101\\
       8092 Z\"{u}rich, Schweiz
       \AND
       \name Sara van de Geer \email geer@ethz.ch \\
      \addr Seminar f\"{u}r Statistik, ETH Z\"{u}rich \\
       R\"{a}mistrasse 101\\
       8092 Z\"{u}rich, Schweiz}

\editor{Arnak Dalalyan}

\maketitle

\begin{abstract}
We study the theoretical properties of image denoising via total variation penalized least-squares. We define the total vatiation in terms of the two-dimensional total discrete derivative of the image and show that it gives rise to denoised images which are piecewise constant on rectangular sets.

 We prove that, if the true image is piecewise constant on just a few rectangular sets, the denoised image converges to the true image at a parametric rate, up to a log factor.
More generally, we  show that the denoised image enjoys oracle properties, that is, it is almost as good as if some aspects of the true image were known.

In other words, image denoising with total variation regularization leads to an adaptive reconstruction of the true image.
\end{abstract}


\section{Introduction}\label{TV2D-JMLR-Section1}

Image denoising is a broad and active field of research (\cite{dabo07,elad10,aria12,zhan18,goya20}), where the aim is to reconstruct an image corrupted with noise. Generally, some assumptions on the structure of the underlying image have to be made to favor denoised images showing such structure (\cite{mamm95,polz03}). One of these assumptions is that the image to reconstruct is constant on few sets belonging to some specific class, as for instance the class of connected sets or the class of rectangular connected sets. Image denoising with total variation regularization is known to promote such piecewise-constant denoised images (\cite{bach11b}).

The use of total variation penalties for image denoising dates back to \cite{rudi92} and has been the subject of various studies (\cite{mamm97-2,cham97,case15,cham17}). For an overview over some theoretical and practical aspects, see \cite{cham10}. The theoretical study of total variation for image denoising has recently experienced a surge of interest (\cite{sadh16,wang16,hutt16,padi17,chat19,fang19}).

\subsection{Review of the literature}

Consider a continuous image $\phi(x,y), (x,y) \in [0,1]^2$ and a discrete or discretized image $f(j,k), (j,k) \in \{1, \ldots, n_1\}\times \{1, \ldots, n_2\}$. In the literature we encounter different definitions of (two-dimensional) total variation. Three of them are listed in what follows.

\begin{itemize}
\item In the seminal work by \cite{rudi92}, total variation is defined in terms of partial derivatives as
$$ \int_0^1 \int_0^1 \sqrt{\left(\frac{\partial}{\partial x} \phi(x,y)\right)^2 + \left(\frac{\partial}{\partial y}\phi(x,y)\right)^2} dx \ dy .$$
Different discretization procedures have been proposed for the total  variation by \cite{rudi92}: isotropic, anisotropic, upwind (\cite{cham11}) and Shannon (\cite{aber17}) total vatiation.
For more details and a recently proposed discretization we refer to \cite{cond17}. 

\item Total variation in  terms of partial derivatives can also be defined as
$$ \int_0^1 \int_0^1 \left|\frac{\partial}{\partial x} \phi(x,y)\right| + \left|\frac{\partial}{\partial y}\phi(x,y)\right| dx \ dy .$$
Its discrete version
$$ \sum_{j=2}^{n_1} \sum_{k=1}^{n_2} \abs{f(j,k)-f(j-1,k)} + \sum_{j=1}^{n_1} \sum_{k=2}^{n_2} \abs{f(j,k)-f(j,k-1)}$$
is considered in \cite{sadh16,wang16,hutt16,chat19} and  corresponds, up to normalization, to summing up the edge differences of the discrete image across a two-dimensional grid graph. Used as a penalty for least squares, this definition results in denoised images which are piecewise constant on connected sets of any shape (\cite{bach11b}).

\item Alternatively, total variation can be defined in terms of the total derivative as
$$ \int_0^1 \int_0^1 \left|\frac{\partial}{\partial x} \frac{\partial}{\partial y}\phi(x,y)\right| dx \ dy $$
and in discretized form as
$$\sum_{j=2}^{n_1} \sum_{k=2}^{n_2} \abs{f(j,k)-f(j-1,k)-f(j,k-1)+f(j-1,k-1)}. $$
This approach is adopted by \cite{mamm97-2}, \cite{fang19} and will also be adopted in this paper. As a penalty for least squares, this definition will be shown to render denoised images which are piecewise constant on rectangular sets.
\end{itemize}

In the literature, the second definition of total variation in terms of partial derivatives is more popular than the one in terms of total derivatives. Least-squares estimators with a penalty on the discrete partial derivatives of the image are the subject of a vast statistical literature. Let $n$ denote the number of pixels of the image. \cite{sadh16} derive minimax rates, which, for large $n$ and under the canonical scaling are of order $\sqrt{\log({n})/n}$. Later, \cite{sadh17b} extend the minimax results to higher order differences and to higher dimensions. \cite{hutt16} prove sharp oracle inequalities with the rate $\log n/ \sqrt{n}$. Lastly, the very recent work by \cite{chat19} focuses  on the constrained  optimization problem (solvable e.g. by \cite{fadi11}) and a tuning-free version thereof. For a certain underlying image a rate faster than the minimax rate is obtained. The approach by \cite{chat19}, as the one used for higher-order total variation regularization (\cite{gunt20}), is based on bounding Gaussian widths of tangent cones.

On the other side, \cite{mamm97-2} define total variation in terms of total derivative, as in this paper, and obtain the rate $n^{-3/5}$ for the estimation of the ``interaction terms''. The same definition of total variation is used by \cite{fang19}, who study a constrained version of the estimator.


\subsection{Contributions}

We prove upper bounds on the mean squared error for image denoising with a total variation penalty promoting piecewise constant estimates on rectangular sets. These upper bounds are presented in the form of oracle inequalities (cf. \cite{kolt06, loun11, dala12a, stuc17, bell17, bell16, bell18, else19}). Oracle inequalities are finite-sample theoretical guarantees on the performance of an estimator, which treat in a unified way both the cases of  well-specified and misspecified models. In particular, we show that the mean squared error of the denoised image is upper bounded by the optimal tradeoff between ``approximation error'' and ``estimation error''. This optimal tradeoff depends on the true underlying image, which is unknown. Hence the term ``oracle'':  the estimator is shown to perform as well as if it knew the aspects of the true image necessary to reach this optimal tradeoff.

We derive oracle inequalities with both fast and slow rates.
\begin{itemize}
\item In the case of fast rates, the estimation error is shown to be of order ${s^*}^{3/2} \log^2 (n) /n$  for oracle images being constant on $s^*$  rectangular sets of roughly the same size. The parametric rate is reached up to the log term and a factor ${s^*}^{1/2}$ due to the two-dimensionality of the problem. The general result with fast rates is Theorem \ref{main.theorem}, while a special case is exposed in Theorem \ref{thm4}. Theorem \ref{main.theorem} is an adaptive result: the bound on the mean squared error of the denoised image depends on the structure in the underlying image.  This dependence is mediated by a so-called oracle, which trades off the fidelity to the underlying image and the number of constant rectangular regions $s^*$ to estimate.
\item In the case of slow rates, the estimation error is shown to be of order $n^{-5/8} \log^{3/8} n$ under the assumption that the total variation of the image is bounded, cf. Theorem \ref{main.slow}. This rate outperforms the rate $n^{-3/5}$ obtained by \cite{mamm97-2}.
\end{itemize}

These contributions build on previous research in the one-dimensional setting, where the classical example is the fused Lasso (\cite{tibs05,dala17,lin17b, orte18}). The term ``fused Lasso''  often refers to the penalty on the total variation of the coefficients in a linear model. Generalizations of the fused Lasso to other graph structures than the chain graph and to penalties on higher-order discrete derivatives are known under the name of edge Lasso (\cite{shar12,hutt16}) and trend filtering (\cite{tibs14,wang16,vand19,gunt20}), respectively.

\subsection{Technical tools}

The skeleton of our proofs (cf. Section \ref{TV2D-JMLR-Section6}) closely follows the proofs of  oracle inequalities for similar estimation problems (cf. the proofs in \cite{hutt16,dala17,orte19-2,vand19}). The more involved part is adapting the techniques previously applied in one dimension to two dimensions, in particular: the derivation of the synthesis form of the estimator, the bound used to control the noise also known as bound on the increments of the empirical process and the bound on the ``effective sparsity''. 
\begin{itemize}
\item We define  image denoising with total variation as an analysis estimation problem: the observations are approximated by a candidate estimator, some aspects of which are penalized. The penalized aspects are computed via a linear operator, the so-called analysis operator, which in our case corresponds to the two-dimensional total derivative operator. \cite{elad07} explain how to obtain a synthesis formulation of analysis estimators. In the synthesis formulation, the candidate estimator is synthesized by a linear combination of atoms. The atoms constitute the moral equivalent of basis vectors (or basis matrices in the case of image denoising). The collection of atoms is called dictionary. The penalty is then enforced on the convex relaxation of the number of atoms used to synthesize the estimator.

As in our previous work in one dimension (\cite{orte18,vand19}), the first step is to reformulate  total variation image denoising in synthesis form and show that the dictionary consists of a collection of indicator functions of half-intervals, see Section \ref{TV2D-JMLR-Section5} and in particular Lemma \ref{expansion.lemma} and \ref{separate-interactions.lemma}. As a consequence, the estimator will be piecewise constant on rectangular regions. Moreover the insights from the synthesis formulation of the estimator will help us in the further analysis of its behavior. 

\item A central step in the derivation of oracle inequalities is to control the random part of the estimation problem consisting of the increments of an empirical process (cf. \cite{vand09a}), whose increments need to be bounded. We apply to the case of image estimation a technique developed by \cite{dala17}. This technique involves the decomposition of the increments of the empirical process into two parts: a part projected onto a suitable linear space and a remainder, see Lemma \ref{2d-appB-l02} in Subsection \ref{TV2D-JMLR-Section6.1}. The projected part will usually be of low rank, while the remainder will contribute to the ``effective sparsity''.

The dictionary atoms of the synthesis formulation are strongly correlated. Thus, even when choosing a low-rank linear subspace spanned by only few dictionary atoms, the remainder will be small. As a crucial consequence, also the contribution to the effective sparsity will be ``small''.

\item The effective sparsity, see Definition \ref{effective.sparsity} in Subsection \ref{TV2D-JMLR-Section6.2} (in vector form) or Definition \ref{effective.sparsity.matrix} in Subsection \ref{interpolating.section} (in matrix form),  measures indicatively the effective number of parameters we have in the model. Indeed,  oracle inequalities with fast rates usually show an estimation error of the order ``effective sparsity $\times \ \log n$/$n$'', and thus the effective sparsity can be interpreted as the effective degrees of freedom that are spent to estimate the model parameters. In this paper, because of the projection arguments used to bound the increments of the empirical process, the effective sparsity will be multiplied by a factor smaller than $\log n$/$n$ to obtain the fast rate. In the literature, the reciprocal of a stronger version of the effective sparsity is also known under the name of ``compatibility constant'', which is related to the restricted eigenvalue (\cite{vand09b,vand16,vand18}). Also in this case, we extend a previously known one-dimensional bound (\cite{vand19}) based on interpolating polynomials to the two-dimensional case. The new bound is based on an interpolating matrix, which interpolates the active parameters and can be found in Lemma \ref{interpolating.lemma} in Subsection \ref{interpolating.section}.

\end{itemize}

\subsection{Organization of the paper}

In Section \ref{TV2D-JMLR-Section2} we introduce the required notation. In Section \ref{TV2D-JMLR-Section3} we define the model and the estimator and we show how an image can be decomposed into global mean, (centered) row and column means and interaction terms. This is a so-called ANOVA decomposition of an image. 
As a preview of the main result, we state in Section \ref{TV2D-JMLR-Section4} a special case for a square image.
In Section \ref{TV2D-JMLR-Section5} we formulate the estimator for the interaction terms in  synthesis form. In Section \ref{TV2D-JMLR-Section6} we expose the standard techniques used to obtain oracle inequalities with fast and slow rates for general analysis problems. The derivation of bounds on the effective sparsity is given in Section \ref{TV2D-JMLR-Section7}, where we also present the details of our main result, which is an oracle inequality with fast rates. In Section \ref{TV2D-JMLR-Section8} we prove the slow rate $n^{-5/8}\log^{3/8}n$.
Section \ref{TV2D-JMLR-Section9} concludes the paper.
\section{Notation and definitions}\label{TV2D-JMLR-Section2}

We expose the mathematical notation required and some basic definitions.

\subsection{Matrix notation}
We model images as matrices of dimension $n_1 \times n_2$ with entries the real-valued pixel values. Let $n:= n_1 n_2$ denote the total number of pixels of an image.

For two integers $i^* \le i$, we use the notation $[i^* : i]=\{i^*, \ldots, i\} $. If $i^*=1$, we write $[i]:= [1 :  i]$.
For a row index $j$ and a column index $k$, $(j,k) \in [ n_1 ] \times [ n_2]$, we refer to the corresponding entries of the matrix $f\in \R^{n_1 \times n_2}$ in two different ways: either by  $f_{j,k}$ using subscripts or by $f(j,k)$ using arguments.

These two equivalent notations will be useful in different situations. For instance, the notation using arguments will come in handy in Section \ref{TV2D-JMLR-Section5}, when deriving the synthesis form of the estimator.

By $\| f \|_2 $ we denote the Frobenius norm of $f\in \R^{n_1 \times n_2}$, that is
$$\| f \|_2 := \biggl ( \sum_{j=1}^{n_1} \sum_{k=1}^{n_2} f_{j,k}^2 \biggr )^{1/2}.$$
Moreover we define
$$ \norm{f}_1:= \sum_{j=1}^{n_1} \sum_{k=1}^{n_2} \abs{f_{j,k}}$$
as  the sum of the absolute values of the entries of $f\in \R^{n_1 \times n_2}$.

\subsection{Total variation}
Let $f\in \R^{n_1 \times n_2}$ be an image. Let $D_1 \in \R^{(n_1-1) \times n_1}$ and $D_2 \in \R^{(n_2-1)\times n_2}$ be discrete difference operators, that is, matrices of the form
$$ \begin{pmatrix*}
-1 & 1  & & \\
 &  \ddots & \ddots &  \\
 & & -1 & 1 
\end{pmatrix*}.$$

\begin{definition}[Two-dimensional discrete derivative operator]
The two-dimensional discrete derivative operator  $\Delta: \R^{n_1 \times n_2} \mapsto \R^{(n_1-1) \times (n_2-1)}$ is defined as
$$\Delta f = D_1  f D_2^{T}.$$
\end{definition}

Note that $\Delta$ is a linear operator and 
$$ ( \Delta f)_{j,k} := f_{j,k}-f_{j,k-1}-f_{j-1,k}+f_{j-1,k-1}, \ (j,k) \in [2: n_1 ] \times [2: n_2].$$

\begin{definition}[Total variation]
The total variation ${\rm TV}(f)$ of an image $f$ is defined as
$${\rm TV} ( f) :=\norm{\Delta f}_1=  \sum_{j=2}^{n_1}  \sum_{k=2}^{n_2}| f_{j,k}-f_{j,k-1}-f_{j-1,k}+f_{j-1,k-1}|.$$
\end{definition}

\subsection{Active set}\label{2dss07}

Fix some set $S\subseteq [3: n_1 -1] \times [ 3 : n_2-1]$. We can think of $S$ as the subset of coefficients of the total derivative $\Delta f$ which are active, that is nonzero. The cardinality of $S$ is denoted by $s:= |S|$. We write 
$ S:= \{ t_1 , \ldots , t_s \}$. We refer to the elements of  $S$ as jump locations. The coordinates of a jump location $t_m$ are denoted by $(t_{1,m},  t_{2,m} ) $, $m=1 , \ldots , s $. Note that we require that
$2 < t_{1,m} < n_1$ and $2< t_{2,m} < n_2$.
This assumption ensures that we have no boundary effects when doing
partial integration, see Lemma \ref{partial-integration.lemma}.

For two  matrices $a = \{ a_{j,k}\}_{ (j,k)\in [ 2 : n_1]\times [2: n_2] }$
and $b = \{ b_{j,k}\}_{ (j,k)\in [ 2 : n_1]\times [2: n_2] }$
we use the symbol $\odot$ for entry-wise multiplication: $( a \odot b )_{j,k} :={a}_{j,k}  b_{j,k} , \ (j,k) \in [2: n_1] \times [2 : n_2] $.

Moreover we define $ a_S :=  \{ a_{j,k}, \ {(j,k)\in S} \}$ and $\ a_{-S}:= \{ a_{j,k}, \ (j,k) \notin S \}$. 
We will use the same notation $a_S\in \R^{(n_1-1)\times (n_2-1)} $ for the matrix
which shares its entries with $a$ for $(j,k) \in S$ and has all its other
entries equal to zero. 
Similarly, $a_{-S}\in \R^{(n_1-1)\times (n_2-1)}$ shares its entries with $a$ for $(j,k) \not \in S$ and has its other entries equal to zero.

\subsection{Linear projections}

For a linear space ${\cal W} $, 
let ${\rm P}_{\cal W} $ denote the projection operator on ${\cal W}$ and
${\rm A}_{\cal W} := I- {\rm P}_{\cal W} $ the corresponding antiprojection operator. The antiprojection operator on ${\cal W}$ computes the residuals of the orthogonal projection on ${\mathcal W}$.

\section{Preliminaries}\label{TV2D-JMLR-Section3}


We want to estimate the image $f^0 \in \R^{n_1 \times n_2}$ based on its noisy observation $ Y = f^0 + \epsilon $,
where $\epsilon \in \R^{n_1 \times n_2}$ is a noise matrix with i.i.d.\  Gaussian entries with known variance $\sigma^2$.
For the case of unknown variance, \cite{orte19-2} show how to simultaneously estimate the signal and the noise variance by extending the idea of the square-root Lasso (\cite{bell11}) to total variation penalized least-squares.

\subsection{ANOVA decomposition of an image}

In this subsection we introduce the ANOVA decomposition, which separates an image $f \in \R^{n_1 \times n_2}$ into four mutually orthogonal components: the global mean, the two matrices of main effects and the matrix of interaction terms. 


\begin{definition}[Global mean]
The {global mean} $f(\circ, \circ)\in \R$ is defined as
$$f(\circ, \circ) := {1 \over n_1 n_2 } \sum_{j=1}^{n_1} \sum_{k=1}^{n_2} f (j,k) . $$
\end{definition}

\begin{definition}[Main effects]
The  {main effects} are defined as $f(\cdot, \circ)= \{f(j,\circ)\}_{(j,k) \in [n_1]\times [n_2]}$ and $f(\circ, \cdot)= \{f(\circ,k)\}_{(j,k) \in [n_1]\times [n_2]}$,
where
$$ f(j, \circ) := {1 \over n_2} \sum_{k=1}^{n_2} f(j,k) - f(\circ, \circ)  , \ j \in [n_1] $$
and
$$ f(\circ, k) := {1 \over n_1} \sum_{j=1}^{n_1} f(j,k)- f (\circ, \circ ), \ k \in [n_2].$$
\end{definition}

Note that $f(\cdot, \circ)$ has identical columns and $f(\circ, \cdot)$ has identical rows.
We define the total variation of the main effects as
$${\rm TV}_1 (f) := \sum_{j=2}^{n_1}| f(j,\circ) - f(j-1, \circ)|$$
and
$${\rm TV}_2 (f) := \sum_{k=2}^{n_2}| f(\circ, k) - f(\circ, k-1)|.$$

\begin{definition}[Interaction terms]
The {interaction terms} are defined as
$$ \tilde f (j,k) = f (j,k) - f(\circ, \circ)  - f(j, \circ) - f(\circ, k) , \ (j,k) \in [n_1 ] \times [n_2].$$
\end{definition}

Let $\psi^{1,1}= \{1\}^{n_1\times n_2}$. The ANOVA decomposition of an image $f$ is
$$ f= f(\circ, \circ)\psi^{1,1}+  f( \cdot, \circ) +f( \circ, \cdot) +  \tilde f $$
and is illustrated in Figure \ref{fig.1} for an image from the Leaf Shape Database.
Note that $f(\circ,\circ)\psi^{1,1}, \ f(\cdot, \circ), \ f (\circ, \cdot)$ and $\tilde f$ are mutually orthogonal and thus we have that
$$ \| f \|_2^2 =n_1 n_2 f^2 (\circ, \circ)   + \| f (\cdot , \circ) \|_2^2 + \| f (\circ , \cdot)  \|_2^2 + \| \tilde f \|_2^2.$$

We now use the ANOVA decomposition to  define the estimator for the interaction terms, which is the main object studied in this paper.

\begin{figure}
\centering
\setkeys{Gin}{width=0.6\textwidth}
\includegraphics{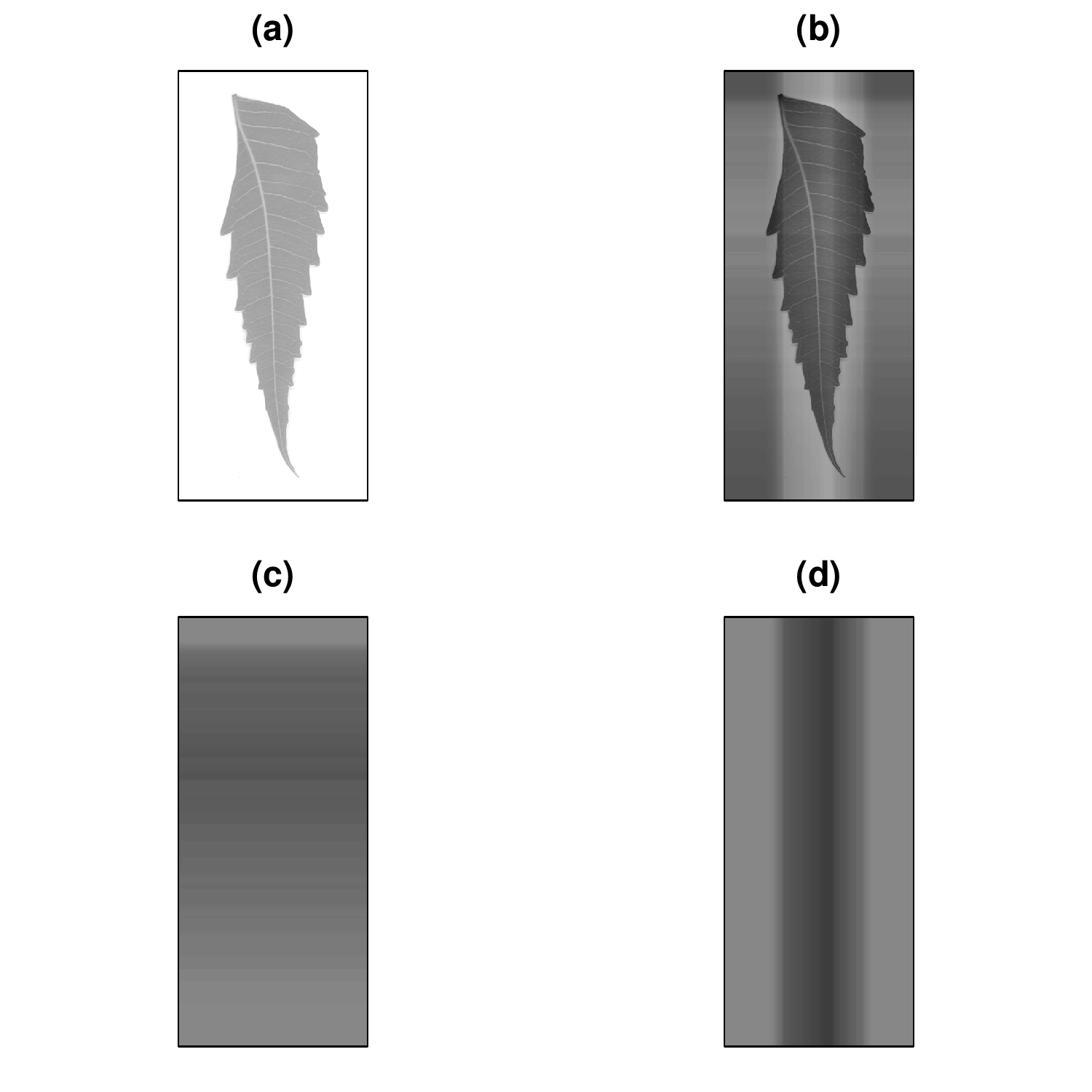}
\caption{The ANOVA decomposition of the image lg1 from the Leaf Shapes Database by \protect\cite{wagh}. The image (a) is the original image $f$, (b) represents the interaction terms $\tilde{f}$ and (c) and (d) are the main effects $f(\cdot,\circ)$ and $f(\circ,\cdot)$, respectively.}\label{fig.1}
\end{figure}

\subsection{The estimator}

We consider the estimator
\begin{equation*} \hat f := \argmin_{{\rm f} \in \R^{n_1 \times n_2} }
\biggl \{ \| Y - {\rm f} \|_2^2 / n + 2 \lambda {\rm TV} ( {\rm f} ) + 2 \lambda_1 {\rm TV} _1 ({\rm f}) +
2 \lambda_2 {\rm TV} _2 ( {\rm f}  )   \biggr \}  ,\end{equation*}
where $\lambda,\lambda_1,\lambda_2>0$ are positive tuning parameters.
We call $\hat f$ the two-dimensional  total variation regularized least squares estimator.  This estimator has the form of an analysis estimator (\cite{elad07}): it approximates the observations under a regularization penalty on the $\ell_1$-norm of a linear operator of the signal $f$.

Since $Y(\circ, \circ)\psi^{1,1}$, $Y( \cdot, \circ)$, $Y ( \circ, \cdot)$ and $\tilde{Y}$ are mutually orthogonal, we may decompose the estimator as
$$ \hat{f}= \hat{f}( \circ, \circ) \psi^{1,1}+ \hat{f}( \cdot, \circ)+ \hat{f}( \circ, \cdot)+ \hat{\tilde{f}},$$
where
\begin{align*}
\hat{f}( \circ, \circ)& := Y( \circ , \circ),\\
\hat{f}( \cdot, \circ)&:= \argmin_{{\rm f}\in \R^{n_1 \times n_2}} \left\{\|{Y(\cdot, \circ)- {\rm f}}\|^2_2/n + 2 \lambda_1 \text{TV}_1 ({\rm f}) \right\}, \\
\hat{f}( \circ, \cdot)&:= \argmin_{{\rm f}\in \R^{n_1 \times n_2}} \left\{\|{Y(\circ, \cdot)- {\rm f}}\|^2_2/n + 2 \lambda_2 \text{TV}_2 ({\rm f}) \right\} ,\\
\hat{\tilde{f}}&:= \argmin_{{\rm f}\in \R^{n_1 \times n_2}} \left\{\|{\tilde{Y}- {\rm f}}\|^2_2/n + 2 \lambda \text{TV} ({\rm f}) \right\} .
\end{align*}

We can also apply the ANOVA decomposition to the underlying image $f^0$:
$$ f^0= f^0(\circ,\circ) \psi^{1,1} + f^0( \cdot, \circ)+ f^0( \circ, \cdot)+ \tilde{f^0}.$$
Then we can estimate $f^0(\circ,\circ)$ by $\hat{f}( \circ, \circ)$, $f^0( \cdot, \circ)$ by $\hat{f}( \cdot, \circ)$, $f^0( \circ, \cdot)$ by $\hat{f}( \circ, \cdot)$ and $\tilde{f^0}$ by $\hat{\tilde{f}}$.

Ordinary least squares is an appropriate method for estimating $f^0(\circ,\circ)$. Indeed the rate of convergence of $Y(\circ,\circ)$ to $f^0(\circ,\circ)$ is $n^{-1}$. The estimation of $f^0( \cdot, \circ)$ and $f^0( \circ, \cdot)$ by ordinary least squares would lead to rates of convergence of order $n_1/n$ and $n_2/n$, respectively. In  this paper we show that both the fast and the slow rate of estimation of $\tilde{f^0}$ by $\hat{\tilde{f}}$ are faster than $n^{-1/2}$, which is the best-case rate of estimation of the main effects by ordinary least squares. 
Without the regularization terms $\lambda_1{\rm TV}_1(f)$ and $\lambda_2{\rm TV}_2(f)$, the speed of estimation of $f^0$ would be limited by the estimation of the main effects.  We therefore propose a regularized method, the so-called fused Lasso (\cite{tibs05}), to estimate both $f^0( \cdot, \circ)$ and $f^0( \circ, \cdot)$ at a faster rate than $n^{-1/2}$. 

Also the noise term $\epsilon$ can be decomposed into four orthogonal components:
$$ \epsilon= \epsilon(\circ,\circ) \psi^{1,1} + \epsilon( \cdot, \circ)+ \epsilon( \circ, \cdot)+ \tilde{\epsilon}.$$
All the four terms of the decomposition present some correlation structure.
This is however not a problem for the analysis of the respective estimators, since the four terms can be seen as as  the projections onto four mutually orthogonal linear subspaces of $\R^{n_1 \times n_2}$.  
Indeed, in the analysis of $\hat{f}( \cdot, \circ)$, the empirical process ${\rm trace}(\epsilon( \cdot, \circ)'f( \cdot, \circ)/n)$ appears. By the idempotence of projection matrices, we have that ${\rm trace}(\epsilon( \cdot, \circ)'f( \cdot, \circ)/n)= {\rm trace}((\epsilon(\circ,\circ)\psi^{1,1}+\epsilon( \cdot, \circ))'f( \cdot, \circ)/n) $, where $\epsilon(\circ,\circ)\psi^{1,1}+\epsilon( \cdot, \circ)$ has rowwise iid entries. Similarly, in the analysis of $\hat{\tilde{f}}$ it holds that ${\rm trace}(\tilde{\epsilon}'\tilde{f}/n)={\rm trace}({\epsilon}'\tilde{f}/n)$.

A slow rate for ${\hat f}(\cdot, \circ)$ and ${\hat f}(\circ, \cdot)$ of order $n^{-2/3}$ is shown in \cite{mamm97-2} using entropy calculations but with large constants and in \cite{orte19-2,vand19} with small constants but an additional logarithmic term.

The adaptivity of the estimators ${\hat f}(\cdot, \circ)$ and ${\hat f}(\circ, \cdot)$ has been established in \cite{lin17b,gunt20,dala17,orte18,orte19-2,vand19}. Let ${\rm s}$ denote the number of jumps in any column of $f^0(\cdot, \circ)$ or the number of jumps in any row of $f^0(\circ, \cdot)$.
We give the fast rates exposed in these papers for the case that the ${\rm s}$ jumps lie on a regular grid.

\cite{lin17b} obtain the rate $\bigo\left({{\rm s}}\left((\log {\rm s}+\log\log n)\log n +\sqrt{{\rm s}}\right)/{n}\right)$,
under the choice of the tuning parameter $\lambda\asymp n^{-\frac{1}{2}} {\rm s}^{-\frac{1}{4}}$, for $n$ large enough.

Under the choice of the tuning parameter $\lambda\asymp \sqrt{\log n/n}$
\cite{dala17} obtain an oracle inequality with the rate $\bigo\left({s\log n}\left({\rm s}+ \log n\right) /{n} \right)$.

Under the choice of the tuning parameter $\lambda\asymp \sqrt{\log n/({\rm s}n)}$, \cite{orte18,orte19-2,vand19} obtain the rate $ \bigo\left({{\rm s}\log^2n}/{n} \right)$,
which is improved by a log term by \cite{gunt20} for a choice of the tuning parameter depending on $f^0$.

Since the estimation of $f^0(\circ,\circ)$, $f^0( \cdot, \circ)$, $f^0( \circ, \cdot)$ can be undertaken in a satisfactory way with estimators already widely studied in the literature,  we are  going to focus on establishing a slow rate and the adaptivity for the estimator of the interaction terms $\hat{\tilde{f}}$. 
For slow rates it will turn out  that the part limiting the speed of estimation of $f^0$ is the estimation of the interaction terms, while for fast rates we will show that  the interaction terms too can be estimated in an adaptive manner.

\section{A taste of the main result}\label{TV2D-JMLR-Section4}

We present our main result, Theorem \ref{main.theorem} from Section \ref{TV2D-JMLR-Section7}, for  the special case of a square image  ($n_1=n_2$) and an active set $S$ defining a regular grid of cardinality $\sqrt{s} \times \sqrt{s}$. To be understood in its generality, Theorem \ref{main.theorem} requires the background knowledge from Section \ref{TV2D-JMLR-Section6}.

\begin{theorem}[Main result for fast rates: a special case]\label{thm4} Let $n_1=n_2$. Let $g \in \R^{n_1 \times n_2}$ be arbitrary. Let $S$ be an arbitrary subset of size $s:=\abs{S}$ of $[3:n_1-1]\times [3:n_2-1]$ defining a regular grid of cardinality $\sqrt{s}\times \sqrt{s}  $ parallel to the coordinate axes.
Choose
$$\lambda \ge 4 \sigma \sqrt{\frac{\log(2n)}{n\sqrt{s}}}.$$
Then, with probability at least $1-1/n$, it holds that
\begin{equation*}
\norm{\hat{\tilde f}-\tilde{f}^0}^2_2/n \le \norm{g-\tilde{f}^0}^2_2/n + 4 \lambda \norm{(\Delta g)_{-S}}_1 + \left(\sigma \sqrt{\frac{s}{n}}+\sigma \sqrt{\frac{2 \log(2n)}{n}}+ \lambda \sqrt{\frac{8s^2 n \log(e^2n)}{(\sqrt{n}-1)^2}} \right)^2.
\end{equation*}
\end{theorem}

If we choose $g=\tilde{f}^0$,  $S$ to be the active set of $\tilde{f}^0$ (given it is a regular grid) and $ \lambda = 4\sigma  \sqrt{{\log(2n)}}/\sqrt{n\sqrt{s}}$, then, with probability at least $1-1/n$, we have that $ \norm{\hat{\tilde f}-\tilde{f}^0}^2_2/n= \bigo\left({s^{3/2}\log^2(n)}/{n} \right)$.
If we instead make the choice $\lambda= 4 \sigma \sqrt{\log(2n)/n}$, which does not depend on $S$, with probability at least $1-1/n$, we have that
$ \norm{\hat{\tilde f}-\tilde{f}^0}^2_2/n= \bigo\left({s^{2}\log^2(n)}/{n} \right)$.

In both cases, the dependence on $s$ is worse than the linear dependence which has been  proven for the one-dimensional case by \cite{dala17,gunt20,vand19}. However, if $s$ is constant the rate is parametric, up to the log factor. \cite{fang19}  prove a similar result restriced to active sets $S$ such that $s=1$.

Theorem \ref{thm4} gives us theoretical guarantees holding for all $g\in \R^{n_1 \times n_2}$ and active sets $S$ defining a regular grid, no matter the structure of $f^0$. Therefore $g$ and $S$ (under some constraints) can be seen as free parameters. The upper bound can be minimized over all $g\in \R^{n_1 \times n_2}$ and all active sets $S$ defining a regular grid. However, minimizing the upper bound requires the knowledge of ${\tilde f}^0$. A pair $(f^*, S^*)$ minimizing the upper bound is called ``an oracle'', since ${\tilde f}^0$ is typically unknown. 
Theorem \ref{thm4} is an oracle inequality in the sense that it guarantees that the (properly tuned) estimator behaves almost as good as if it would know the aspects of ${\tilde f}^0$ required to minimize the upper bound and optimally trade off all of its terms.

Theorem \ref{thm4} is also an adaptive result: the estimator $\hat{\tilde f}$ is shown to adapt to the underlying image ${\tilde f}^0$, in particular to the number and location of its jumps. The adaptation to ${\tilde f}^0$  is achieved by means of the optimal tradeoff between the approximation of ${\tilde f}^0$ by the oracle $f^*$ and the almost parametric rate of estimation of the rectangular pieces defined by the oracle active set $S^*$.

We will expose the more general version of this theorem holding for active sets $S$  not necessarily defining a regular grid in Section \ref{TV2D-JMLR-Section7}.

\begin{figure}[h]
\centering
\includegraphics[width=\textwidth]{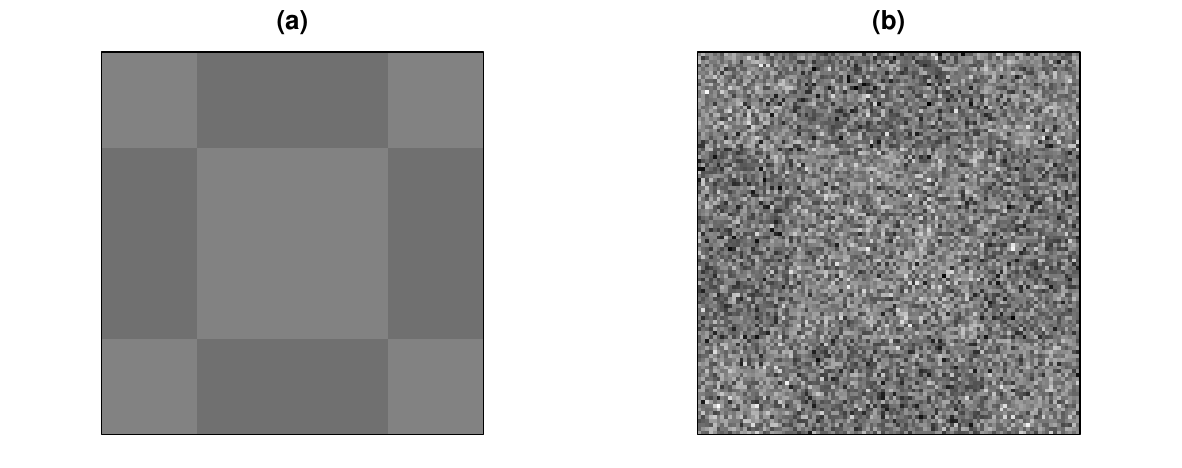}
\caption{Plot {\bf (a)} shows ${\tilde f}^0$ for $f^0$ as in Equation \eqref{eq.truth}. Plot {\bf (b)} shows ${\tilde Y}={\tilde f^0}+{\tilde \epsilon}$ for a realization of ${\tilde \epsilon}$ with $\sigma=1$.}
\label{fig.truth}
\end{figure}

For $n_1=n_2$ being a multiple of 4 consider the image $f^0 \in \R^{n_1 \times n_2}$ defined as
\begin{equation}\label{eq.truth}
f^0_{j,k}=1_{\{n_1/4+1\le j \le 3n_1/4\}} 1_{\{n_2/4+1 \le k \le 3n_2/4\}},\ (j,k) \in [n_1]\times [n_2].
\end{equation}

Figure \ref{fig.truth} shows ${\tilde f}^0$ and ${\tilde Y}={\tilde f^0}+{\tilde \epsilon}$ for $n_1=100$ and $\sigma=1$. Figure \ref{fig.sim} shows some simulations result for denoising the image ${\tilde Y}={\tilde f^0}+{\tilde \epsilon}$ with $\sigma=1$, where, for $n_1=n_2 \in \{4, 8,  \ldots, 196,200\}$, $f^0$ is taken as in Equation \eqref{eq.truth}. For such images, $\Delta f^0$ has 4 nonzero components. Therefore we chose $s=4$. The estimator was computed via a detour through its synthesis formulation (see Section \ref{TV2D-JMLR-Section5}), which allowed to use the \textsf{R} package \texttt{glmnet}.

\begin{figure}[h]
\centering
\includegraphics[width=\textwidth]{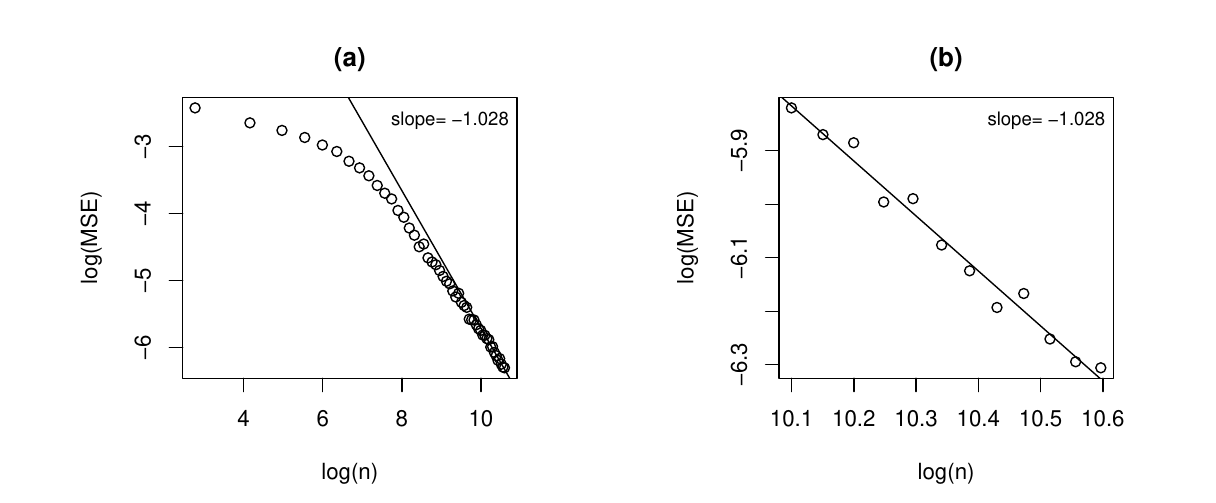}
\caption{Plot {\bf (a)} displays  the logarithm of the average mean squared error of the estimator over 40 realizations of the noise term versus $\log(n)$, for $n_1=n_2 \in \{4, 8,  \ldots, 196,200\}$. The least squares fit with slope $-1.028$ is based on the values for $n_1=n_2 \in \{156, \ldots, 200\}$ shown in detail in plot {\bf (b)}.}
\label{fig.sim}
\end{figure}


The results of the simulation support our findings: if tuned with $\lambda= \sqrt{\log(2n)/(2n)}$, the estimator $\hat{\tilde f}$ converges at an almost parametric rate to the underlying piecewise rectangular image ${\tilde f}^0$. However, the rate of convergence $n^{-1}$ (up to log terms) is achieved only for $n$ large enough and the tuning parameter has to be chosen smaller by a constant factor than the smallest theoretical choice $\lambda= 4 \sqrt{\log(2n)/(2n)}$ suggested by Theorem \ref{thm4}.
\section{Synthesis form}\label{TV2D-JMLR-Section5}
Recall the analysis estimator 
$$\hat{\tilde{f}}:= \argmin_{{\rm f}\in \R^{n_1 \times n_2}} \left\{\|{\tilde{Y}- {\rm f}}\|^2_2/n + 2 \lambda {\rm TV} ({\rm f}) \right\}.$$
Analysis estimators approximate the observations under a penalty on the norm of a linear operator -- a so-called analysis operator -- applied to the candidate estimator ${\rm f}$, in this case ${\rm TV}({\rm f})=\norm{\Delta {\rm f}}_1$.

Analysis estimators can be rewritten as  synthesis estimators (\cite{elad07}).
A well-known instance of synthesis estimator is the Lasso (\cite{tibs96}).
The synthesis approach to  estimation is constructive: the signal is approximated by a linear combination of atoms under a penalty on the norm of the coefficients of this linear combination.
By looking at the properties of the collection of atoms -- the so-called dictionary -- one can gain some insights into the structure of the estimator. 

In our case, the synthesis formulation shows that the estimator $\hat{\tilde{f}}$  produces piecewise rectangular estimates.
Moreover,  the synthesis formulation of $\hat{\tilde{f}}$ will be of great help in computing  ``bounds on the antiprojections'' (see Definition \ref{def.bound.antiproj} and Lemmas \ref{antiprojections.lemma} and \ref{mesh.antiproj}), which  are essential ingredients of Lemma \ref{2d-appB-l02} to control the increments of the empirical process.
For a detailed discussion on the relation between analysis and synthesis estimators we refer to \cite{elad07} and to \cite{orte19-1}, who focus on analysis and synthesis in total variation regularization.


We first express a matrix $f$ as linear combination of dictionary matrices.
We then show that $\hat{\tilde f}$ can be written as a synthesis estimator using these dictionary matrices.
Here the notation with arguments instead of subscripts comes in handy. 

Consider some $f \in \R^{n_1 \times n_2}$.
We may write for $j \in [n_1]$ and $k \in [n_2] $,
$$ f (j,k)  = \sum_{j^{\prime}=1}^{n_1} \sum_{k^{\prime}=1}^{n_2} \beta_{j^{\prime}, k^{\prime}} 
\psi^{j^{\prime}, k^{\prime}} (j,k) ,$$
where for $( j^{\prime}, k^{\prime} )\in [n_1]\times[n_2]$ the dictionary matrices are $\psi^{j^{\prime}, k^{\prime}}$ with
$$ \psi^{j^{\prime}, k^{\prime}} (j,k) = 1_{ \{ j \ge j^{\prime}, k \ge k^{\prime} \}}, \ (j,k)\in [n_1]\times[n_2]  $$
and
$$\beta_{j^{\prime},k^{\prime}}:=\begin{cases} f(1,1) ,& (j^{\prime} k^{\prime})=(1,1), \\
f(j',1)-f(j'-1,1),&  (j^{\prime},k^{\prime}) \in
[2: n_1] \times [1], \\
 f(1,k')-f(1,k'-1) ,& (j^{\prime}, k^{\prime}) \in [1]\times [2: n_2] , \\
( \Delta f)_{j^{\prime}, k^{\prime} } 
,&  (j^{\prime} ,k^{\prime} ) \in [2: n_1] \times [2 :n_2] . \end{cases}$$

We call the collection of matrices $\{\psi^{j^{\prime}, k^{\prime}}\}_{(j',k')\in [n_1]\times [n_2]} $ the dictionary. The dictionary consists of a collection of indicator functions of half intervals. Therefore, a sparse linear combination of elements of the dictionary will be piecewise constant on rectangular sets.

Define 
$$\tilde \psi^{j,k}:=\begin{cases}
\psi^{1,1},&   (j,k)=(1,1), \\
 \psi^{j,1} -  \psi^{j,1} (\circ , \circ)= {\rm A}_{{\rm span}(\psi^{1,1})} \psi^{j,1}   ,& (j,k)\in [2:n_1] \times [1],  \\
\psi^{1,k} -  \psi^{1, k}  (\circ, \circ)={\rm A}_{{\rm span}(\psi^{1,1})} \psi^{1,k}   ,& (j,k)\in [1]\times [2:n_2], \\
 \psi^{j,k} -  \psi^{j , k} (\cdot, \circ)    - \psi^{j , k} (\circ, \cdot)
 - \psi^{j,k} (\circ, \circ)& \\
 \ \ ={\rm A}_{{\rm span}(\{\psi^{j,1}\}_{j \in [n_1]}, \{\psi^{1,k}\}_{k \in [n_2]} )} \psi^{j,k}  , & (j,k) \in [2:n_1]\times [2:n_2]   . \end{cases}$$

The four resulting linear spaces $ {\rm span}(\tilde \psi^{1,1})$, ${\rm span}(\{\tilde \psi^{j,1}\}_{j \in [2:n_1]})$, ${\rm span}(\{\tilde \psi^{1,k}\}_{k \in [2:n_2]})$ and ${\rm span}(\{\tilde \psi^{j,k}\}_{(j,k) \in [2:n_1]\times [2:n_2]})$  are mutually orthogonal. Moreover, the atoms of the dictionary $\{\tilde{\psi}^{j,k}\}_{(j,k)\in [n_1]\times[n_2]}$ are piecewise constant on rectangular sets.

 Lemma \ref{expansion.lemma} gives the form of the coefficients needed to express an image $f$ as a linear combination of  the  matrices $\{\tilde \psi^{j,k}\}_{(j,k)\in [n_1]\times [n_2]}$.
 
 \begin{lemma}[Construct a piecewise rectangular image] \label{expansion.lemma}
 It holds that
 $$ f= \sum_{j=1}^{n_1} \sum_{k=1}^{n_2} \tilde \beta_{j,k} \tilde \psi^{j,k} , $$
 where
 $$ 
 \tilde \beta_{j,k} =\begin{cases} f ( \circ, \circ), & (j,k)=(1,1), \\
f ( j, \circ)-f ( j-1, \circ) ,&   (j,k) \in [2: n_1]\times [1] , \\
 f ( \circ, k )-f ( \circ, k -1) ,&   (j,k) \in [1] \times [2: n_2] , \\
 (\Delta f)_{j,k} ,&   (j,k) \in [ 2:n_1]\times [2: n_2]  . 
 \end{cases}$$
 \end{lemma}
 \begin{proof}
See Appendix \ref{proof.expansion.lemma}.
 \end{proof}

The next lemma, based on Lemma \ref{expansion.lemma}, gives a synthesis form of the estimator $\tilde{ \hat{f}}$.

\begin{lemma}[Synthesis formulation]\label{separate-interactions.lemma} We have
$$ \hat {\tilde f} = \sum_{j=2}^{n_1} \sum_{k=2}^{n_2}
\hat { \tilde \beta}_{j,k} \tilde \psi^{j,k}, $$
where
\begin{equation*}
\hat {\tilde \beta }_{j,k} = \argmin_{ \{ \beta_{j,k} \}_{(j,k)\in [2:n_1]\times [2:n_2]} } \biggl \{ \| Y - \sum_{j=2}^{n_1}  \sum_{k=2}^{n_2}  \beta_{j,k} \tilde \psi^{j,k} \|_2^2/n +2 \lambda \sum_{j=2}^{n_1}  \sum_{k=2}^{n_2}   | \beta_{j,k} |  \biggr \} .
 \end{equation*}
\end{lemma} 
\begin{proof}
See Appendix \ref{proof.separate-interactions.lemma}.
\end{proof}

Note that $\sum_{j=2}^{n_1} \sum_{k=2}^{n_2}  | \hat{\tilde{\beta}}_{j,k} | = {\rm TV} (\hat{\tilde f} )= {\rm TV} (\hat{ f} )$.

\section{Oracle inequalities}\label{TV2D-JMLR-Section6}

In this section we expose standard techniques used to derive oracle inequalities with fast and slow rates. We closely follow \cite{hutt16, dala17,vand19, orte19-2}.
This section can therefore  be viewed as a preparatory step, which frames the work that has to be done in order to  establish  adaptivity as well as the rate $n^{-5/8}\log^{3/8}n$ for the estimator of the interaction terms.

Indeed, we do not yet exploit the specific properties of the two-dimensional total derivative operator $\Delta$. These properties will be further explored in Sections \ref{TV2D-JMLR-Section7} and \ref{TV2D-JMLR-Section8}. The current section can be seen as the background knowledge already present in the literature. It is complemented by our results in Sections \ref{TV2D-JMLR-Section7} and \ref{TV2D-JMLR-Section8}, which are new and specific for total variation image denoising.

For simplicity, in this section we look at matrices as if they were vectors by concatenating their entries by columns. We define the dictionary
$$ \tilde{\Psi}:= \{{\tilde \psi}^{j,k} \}_{(j,k)\in \{2, \ldots, n_1\}\times \{2, \ldots, n_2\}} \in \R^{n_1n_2 \times (n_1-1)(n_2-1)}$$
and the two-dimensional discrete derivative operator
$$ \Delta:= \begin{pmatrix*} D_1 \otimes D_2  \end{pmatrix*} \in \{-1,0,+1\}^{(n_1-1)(n_2-1)\times n_1 n_2},$$
where $\otimes$ denotes the Kronecker product.
Note that the dictionary $\tilde{\Psi}$ is the remainder of the projection of the last $n_1 n_2 -n_1 - n_2 +1$ columns of the dictionary
$$ {\Psi}:= \{{ \psi}^{j,k} \}_{(j,k)\in [n_1]\times [n_2]} \in \R^{n_1n_2 \times n_1n_2}$$
onto its first $n_1+n_2-1$ columns. 

Recall the estimator
$$ \hat{\tilde f}= \arg \min_{{\rm f} \in \R^{n}} \left\{ \norm{{\tilde Y} -{\rm f}}^2_2/n + 2 \lambda \norm{\Delta {\rm f}}_1 \right\}, \lambda >0.$$

To guarantee favorable error bounds the tuning parameter $\lambda$ has to be chosen carefully. It has to be chosen large enough to overrule the noise, but not too large.

A choice of the tuning parameter $\lambda$ that guarantees that all the noise is overruled is the ``universal choice''
 $$ \lambda_0 (t) :=\sigma \sqrt { 2 \log (2n) +2 t \over n }, t >0 . $$
This choice results from the assumption that all noise has to be overruled by the penalty: the the structure encoded in the analysis operator $\Delta$ and in the active set $S$ is not considered.
The projection arguments by \cite{dala17} exposed in Lemma \ref{2d-appB-l02} in Subsection \ref{TV2D-JMLR-Section6.1} take into account the structure encoded in $\Delta$ and show that only the part of the noise not being correlated with the candidate structure of the estimator  needs to be overruled by the tuning parameter $\lambda$. Thus, more favorable error bounds can be obtained with a choice of $\lambda$ smaller than the universal choice $\lambda_0(t)$. How much $\lambda_0(t)$ has to be downscaled depends then on the correlation in the structure encoded in $\Delta$ and $S$.

The universal choice $\lambda_0(t)$ can therefore be seen as a worst-case choice, which always overrules the noise. It always does its job, but not always the best job.


The following inequality is the starting point for the proof of  oracle inequalities with both fast and slow rates and can be found for instance in \cite{vand16,hutt16,orte18,vand19,orte19-2}.

\begin{lemma}[Basic inequality]\label{2d-appB-l01}
For all $g\in \R^n$ we have that
\begin{equation*}
\norm{\hat{\tilde f}- {\tilde f^0}}^2_2/n + \norm{\hat{\tilde f}- g}^2_2/n \le \norm{g- {\tilde f^0}}^2_2/n +2 \frac{{\tilde \epsilon}^{T}(\hat{\tilde f}-g)}{n}+ 2 \lambda (\norm{\Delta g}_1- \norm{\Delta \hat{\tilde f}}_1).
\end{equation*}
\end{lemma}

\begin{proof}
See Appendix \ref{proof.basic.inequ}.
\end{proof}

\subsection{Bounding the increments of the empirical process}\label{TV2D-JMLR-Section6.1}

Let $S \subseteq [(n_1-1)(n_2-1)]$. Let $\tilde{\Psi}_{i}$ denote the $i^{\rm th}$ column of $\tilde \Psi$. We write $\tilde{\Psi}_{S}:= \{{\tilde \Psi}_i\}_{i \in S}$ and $\tilde{\Psi}_{-S}:= \{{\tilde \Psi}_i\}_{i \not\in S}$. Denote by ${\rm P}_{S}:= {\tilde \Psi}_{S} ({\tilde \Psi}_{S}^{T} {\tilde \Psi}_{S})^{-1} {\tilde \Psi}_{S}^{T}$ the orthogonal projection matrix onto the column span of $\tilde{\Psi}_{S}$ and by ${\rm A}_{S}:= {\rm I}_n-{\rm P}_{S}$ the corresponding antiprojection matrix.

Empirical processes and their relevance for statistics are discussed for instance in \cite{vand07b,vand09a}.
In this subsection we are going to expose a high-probability upper bound for the increments of the empirical process by projection arguments proposed by \cite{dala17}.

The increments of the empirical process we study are given by
$$\left\{\frac{{\tilde \epsilon}^{T} f}{n}: f \in \R^n\right\}= \left\{\frac{{\tilde \epsilon}^{T} {\tilde f}}{n}: f \in \R^n\right\}= \left\{\frac{{ \epsilon}^{T} {\tilde f}}{n}: f \in \R^n\right\},$$
where the equality holds because of the idempotence of projection matrices. The basis of the techinque to bound the increments of the empirical process by \cite{dala17} is to decompose them into a part projected onto a low-rank linear space and a remainder, the so-called antiprojection:
\begin{equation}\label{Eq.decompose} \frac{{\tilde \epsilon}^{T} {\tilde f}}{n} = \frac{{\tilde \epsilon}^{T} {\rm P}_{S} {\tilde f}}{n} + \frac{{\tilde \epsilon}^{T} {\rm A}_{S}{\tilde f}}{n} .\end{equation}

We now define the bound on the antiprojections, the inverse scaling factor and the noise weights, which are needed to control the increments of the empirical process by projection arguments.

\begin{definition}[Bound on the antiprojections]\label{def.bound.antiproj}
A bound on the antiprojections $\tilde{v} \in \R^{(n_1-1)(n_2-1)}$ is a vector (or matrix), such that

$${\tilde v}_i \ge \norm{({\rm I}_n-{\rm P}_{S}) {\tilde \Psi}_i}_2/ \sqrt{n}, \ \forall i \in [(n_1-1)(n_2-1)].$$
\end{definition}

Based on the bound on the antiprojections $\tilde v$ we define the inverse scaling factor and the noise weights, which will be important in determining the choice of the tuning parameter $\lambda$ and the bound on the effective sparsity, respectively.

\begin{definition}[Inverse scaling factor]
Let $\tilde v$ be a bound on the antiprojections. The inverse scaling factor $\tilde{\gamma} \in \R$ is defined as
$\tilde{\gamma}:= \norm{\tilde{v}_{-S}}_{\infty}$.
\end{definition}

The inverse scaling factor ${\tilde \gamma}$ depends on the analysis operator $\Delta$ via dictionary ${\tilde \Psi}$ and on the active set $S$. 

\begin{definition}[Noise weights]
Let $\tilde v$ be a bound on the antiprojections and $\tilde \gamma$ the corresponding inverse scaling factor. The noise weights $v\in \R^{(n_1-1)(n_2-1)}$ are defined as
 $v:= {\tilde v}/ {\tilde \gamma} \in [0,1]^{(n_1-1)(n_2-1)}$.
\end{definition}

The following lemma is inspired by the proof of Theorem 1 in \cite{dala17} and can be found in a more general form as Lemma A.2 in \cite{vand19}.

\begin{lemma}[Control the increments of the empirical process with projections]\label{2d-appB-l02}
For $x,t>0$ choose 
$$\lambda \ge \tilde{\gamma} \lambda_0(t).$$
Then, $\forall f \in \R^{n_1n_2}$, with probability at least $1-e^{-x}-e^{-t}$ it holds that
$$ \frac{{\tilde \epsilon}^{T} f}{n} \le \frac{\norm{f}_2}{\sqrt{n}}\left(\sigma\sqrt{\frac{2x}{n}} + \sigma \sqrt{\frac{s}{n}} \right) + \lambda \norm{{v}_{-S} \odot (\Delta f)_{-S}}_1.$$
\end{lemma}

\begin{proof}
See Appendix \ref{bound.ep.proof}.
\end{proof}

Lemma \ref{2d-appB-l02} can be interpreted as a bound on the increments of the empirical process tailored to the structure of the estimation problem. Indeed, the linear space onto which the noise is projected is chosen depending on the analysis operator $\Delta$ and on the candidate active set $S$. As a consequence of the projection arguments used in its proof, one can choose the tuning parameter smaller than the universal choice $ \lambda_0(t)$ by a factor ${\tilde \gamma}$, which depends on the structure encoded in $\Delta$ (or ${\tilde \Psi}$) and $S$. The universal choice of the tuning parameter is retrieved by choosing $S=\emptyset$, which is equivalent to neglecting all structure.

\subsection{Fast rates}\label{TV2D-JMLR-Section6.2}

Oracle inequalities with fast rates are characterized by the presence of the so-called effective sparsity in the upper bound. To define the effective sparsity we need the notion of sign configurations. Indeed, we will apply the definition of effective sparsity to an image, the signs of whose jumps we do not know. Therefore we look for a bound on the effective sparsity holding for all sign configurations.

\begin{definition}[Sign configuration]
Let $q \in [-1,1]^{(n_1-1)(n_2-1)}$ be s.t.
$$q_{i}\in\begin{cases} \{-1,+1\},&  i\in S, \\  [-1,1], & i\notin S.\end{cases}$$
We call  $q_S \in \{-1,0,1\}^{(n_1-1)(n_2-1)}$ a sign configuration.
\end{definition}

We now define the effective sparsity as in \cite{vand19}. This definition will be reformulated in matrix form in Definition \ref{effective.sparsity.matrix}  in Section \ref{TV2D-JMLR-Section7}.

\begin{definition}[Effective sparsity]\label{effective.sparsity}
 Let $S $ be an active set, $q_S \in \{-1,0,1\}^{(n_1-1) (n_2-1)}$ be a sign configuration and $v \in [0,1]^{(n_1-1) (n_2-1)}$ be noise weights.
The effective sparsity $\Gamma ( S, {v}_{-S},q_S )\in \R$ is defined as
$$\Gamma ( S, {v}_{-S},q_S ) = \max  \{   q_S^{T}  (\Delta f)_S  - \| (1- {v})_{-S}  \odot ( \Delta f)_{-S} \|_1  :\| f \|_2^2 / n=1 \}.$$
Moreover we write 
$$ \Gamma ( S, {v}_{-S}):= \max_{q_S}\ \Gamma ( S, {v}_{-S},q_S ).$$
\end{definition}

The definition of effective sparsity consists of two parts: a first term representing approximately the number of jumps -- the sparsity -- and a second term which is a discount due to the correlation of the non-active dictionary atoms with the active ones. Hence the name ``effective sparsity''. The larger this correlation, the larger the discount for the effective sparsity.

An oracle inequality with fast rate is shown in the following theorem, which corresponds to Theorem 2.1 in \cite{orte19-2} and to the adaptive bound of Theorem 2.2 in \cite{vand19}.

\begin{theorem}[Oracle inequality with fast rates]\label{2d-appB-thm3}
Let $g \in \R^n$ and $S \subseteq [(n_1-1)(n_2-1)]$ be arbitrary.
For $x,t>0$, choose  $\lambda \ge {\tilde \gamma} \lambda_0(t)$.
Then, with probability at least $1-e^{-x}-e^{-t}$, it holds that 
\begin{equation*}
\norm{\hat{\tilde f}-{\tilde f^0}}^2_2 /n \le  \norm{g-{\tilde f^0}}^2_2 /n + 4 \lambda \norm{(\Delta g)_{-S}}_1 + \left( \sigma\sqrt{\frac{2x}{n}} +\sigma \sqrt{\frac{s}{n}} + \lambda \Gamma(S, v_{-S},q_S) \right)^2,
\end{equation*}
where $q_S= \text{sign}((\Delta g)_{S})$.
\end{theorem}

\begin{proof}
See Appendix \ref{fast.rate.proof}.
\end{proof}

The fast rate of Theorem \ref{2d-appB-thm3} is given by $\lambda^2\Gamma^2(S,{v}_{-S})\asymp {\log(n){\tilde \gamma}^2\Gamma^2(S,{v}_{-S})}/{n}$. Typically, we expect $\Gamma^2(S,{v}_{-S})$ to scale approximately as $\Gamma^2(S,{v}_{-S})\asymp s/{\tilde \gamma}^{2}$.  We will prove in Lemma \ref{Gamma.lemma} in Section \ref{TV2D-JMLR-Section7} that for image denoising with total variation regularization the effective sparsity scales  as $\Gamma^2(S,{v}_{-S})\asymp s^{3/2} \log(n)/{\tilde \gamma}^{2} $. We have an extra factor $s^{1/2}$ due to the two-dimensional nature of the problem and a log factor due to the noise.

Using Lemma \ref{2d-appB-l02} to bound the increments of the empirical process has two effects:
on the one side we can choose a tuning parameter $\lambda= {\tilde \gamma}\lambda_0(t)$ smaller than the universal choice $\lambda_0(t)$. Thus, the rate that would be obtained with $\lambda= \lambda_0(t)$ can be obtained with $\lambda= {\tilde \gamma}\lambda_0(t)$ and a bound on the effective sparsity larger by a factor $1/{\tilde \gamma}^2$. On the other side, the effective sparsity is increased by an additive $\norm{{v}_{-S}  \odot ( \Delta f)_{-S}}_1$.

To prove adaptivity,  we need to find an appropriate bound on the antiprojections $\tilde v$, the corresponding scaling factor $\tilde \gamma$, the noise weights $v$ and finally prove a bound on the effective sparsity $\Gamma( S, {v}_{-S},q_S )$ holding for all sign confgurations $q_S$. This will be the topic of Section \ref{TV2D-JMLR-Section7}.

\subsection{Slow rates}

The next theorem corresponds to  Theorem 2.2 in \cite{orte19-2} and to the non-adaptive bound of Theorem 2.2 in \cite{vand19}.
\begin{theorem}[Oracle inequality with slow rates]\label{2d-appB-thm4}
Let $g \in \R^n$ and $S \subseteq [(n_1-1)(n_2-1)]$ be arbitrary.
For $x,t>0$, choose $\lambda \ge {\tilde \gamma} \lambda_0(t)$.
Then, with probability at least $1-e^{-x}-e^{-t}$, it holds that 
$$
\norm{\hat{\tilde f}-{\tilde f^0}}^2_2 /n \le \norm{g-{\tilde f^0}}^2_2 /n + 4 \lambda \norm{\Delta g}_1 + \left( \sigma \sqrt{\frac{2x}{n}} + \sigma \sqrt{\frac{s}{n}}  \right)^2.$$
\end{theorem}

\begin{proof}
See Appendix \ref{slow.rate.proof}.
\end{proof}

To obtain the rate $n^{-5/8}\log^{3/8}n $, we need to choose $S$ in a way that optimally trades off the term $s/n$ and the term $\tilde \gamma \lambda_0(t) \norm{\Delta g}_1$. This will be the topic of Section \ref{TV2D-JMLR-Section8}.
\section{Adaptive rates for image denoising}\label{TV2D-JMLR-Section7}

Our objective for this section is to establish that $\hat f$ can adapt to the
number of jumps in the main effects and the interaction terms. The main effects can be
dealt with by using the results for the one-dimensional total variation regularized estimator (see \cite{dala17,gunt20,vand19}).
Thus, our main result will be to show that the estimator $\hat{\tilde f}$  of the interaction terms  is adaptive in that it can adapt to the underlying true interaction terms ${\tilde f}^0$. We will prove an upper bound on the mean squared error of $\hat{\tilde f}$ which can be different for different values of ${\tilde f}^0$. In practice ${\tilde f}^0$ is unknown. However adaptivity guarantees that the estimator can ``sense'' different structures in the underlying ${\tilde f}^0$ and adapt to them.

To prove adaptivity, we need to establish a bound for the so-called effective sparsity (Definition \ref{effective.sparsity.matrix}),
which in turn can be derived by using interpolating matrices (see Lemma \ref{interpolating.lemma}). This way of bounding the effective sparsity is an extension to the two-dimensional case of the bound on the effective sparsity for one-dimensional total variation regularized estimators based on interpolating vectors exposed in \cite{vand19}.
The combination of the new bound on the effective sparsity (see Lemma \ref{Gamma.lemma}) with the standard Theorem \ref{2d-appB-thm3} will lead to our main result.

The roadmap for this section is the following: in Subsection \ref{main.result.section} we will state our main result. The result follows by combining the general oracle inequality for analysis estimators given in Theorem \ref{2d-appB-thm3} with the results of Subsections \ref{interpolating.section}-\ref{noise.section} and will be proved in the conclusive Subsection \ref{proof.main.thm.section}. In Subsection \ref{interpolating.section} we define interpolating matrices and show how to carry out discrete partial integration in two dimensions. In Subsection \ref{effective-sparsity.section} we prove a bound on the effective sparsity and in Subsection \ref{noise.section} we show how to find suitable noise weights.

\subsection{Main result}\label{main.result.section}
We present the main result: an oracle inequality for the estimator 
 $\hat {\tilde f} $ of the interaction terms $\tilde{f}^0$.
 
We fix an active set $S \subseteq [3 : n_1-1 ] \times [3: n_2-1]$. The discussion that follows, and in particular  also Theorem \ref{main.theorem},  depends on the choice of $S$, which can therefore be considered as a ``free parameter''.
 
Given an active set $S \subseteq [3 : n_1-1 ] \times [3: n_2-1]$, we can partition $[2: n_1 ] \times [2: n_2]$ into $s$ subsets, consisting of the points closest to $t_m$, $m=1 , \ldots , s$ with respect to the city block metric. This corresponds to a Voronoi tessellation. 
However, a Voronoi tessellation typically has subsets of relatively irregular shape. We 
will require that the partition consists of rectangles to ease the construction of an interpolating matrix. 
The concept of interpolating matrix is presented in  Section \ref{interpolating.section} and will be applied in the bound for the effective sparsity in Section \ref{effective-sparsity.section}.

\begin{definition}[Rectangular tessellation]\label{def.rectangular.tess}  We call $\{ R_m \}_{m=1}^s $ a rectangular tessellation of $[2: n_1 ] \times [2: n_2] $
if it satisfies the following conditions:\\
$\bullet$ each $R_m \subseteq [2:n_1]\times[2:n_2]  $ is a rectangle ($m=1 , \ldots , s $);\\
$\bullet$ $ \cup_{m=1}^s R_m = [2: n_1 ] \times [ 2: n_2] $;\\
$\bullet$ for all $m $ and $m^{\prime}\not= m $, the rectangles $R_{m}$ and $R_{m^{\prime}} $ ($m \not=m^{\prime}$) possibly share
boundary points, but not interior points;\\
$\bullet$ for all $m$, the jump location $t_m$ is an interior point of $R_m $.
\end{definition}

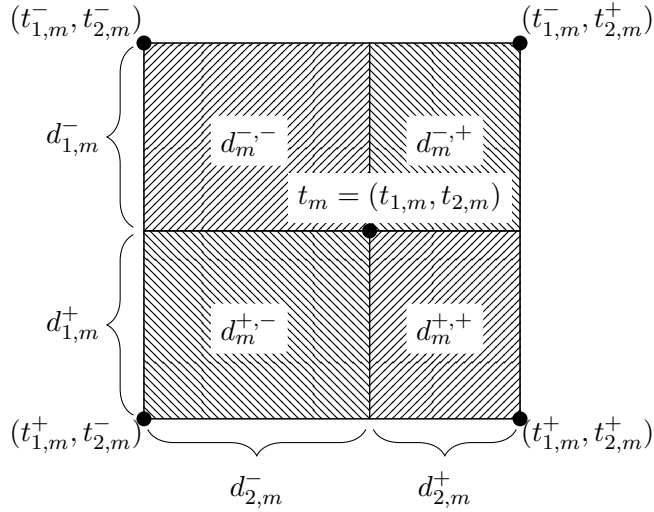
\begin{figure}[h]
\centering
\begin{tikzpicture}
\draw (0,0) rectangle (10/2,10/2);
\draw (0,5/2) -- (10/2,5/2);
\draw (6/2,0) -- (6/2,10/2);

\draw[pattern=north west lines] (0,0) rectangle (6/2,5/2);
\draw[pattern=north west lines] (6/2,5/2) rectangle (10/2,10/2);
\draw[pattern=north east lines] (0,5/2) rectangle (6/2,10/2);
\draw[pattern=north east lines] (6/2,0) rectangle (10/2,5/2);

\node[fill=white,rounded corners=2pt] at (2.8/2,2.4/2) {$d_m^{+,-}$};
\node[fill=white,rounded corners=2pt] at (2.8/2,7.4/2) {$d_m^{-,-}$};
\node[fill=white,rounded corners=2pt] at (8/2,7.4/2) {$d_m^{-,+}$};
\node[fill=white,rounded corners=2pt] at (8/2,2.4/2) {$d_m^{+,+}$};

\node[circle,inner sep=2pt,fill=black,label={[shift={(-.9,-.7)}]$(t_{1,m}^+,t_{2,m}^-)$}] at (0,0) {};
\node[circle,inner sep=2pt,fill=black,label={[shift={(.9,-.7)}]$(t_{1,m}^+,t_{2,m}^+)$}] at (10/2,0) {};
\node[circle,inner sep=2pt,fill=black,label={[shift={(0.9,-.2)}]$(t_{1,m}^-,t_{2,m}^+)$}] at (10/2,10/2) {};
\node[circle,inner sep=2pt,fill=black,label={[shift={(-.9,-.2)}]$(t_{1,m}^-,t_{2,m}^-)$}] at (0,10/2) {};

\node[circle,inner sep=2pt,fill=black,label={[shift={(.8/2,0)},fill=white]$t_m=(t_{1,m},t_{2,m})$}] at (6/2,5/2) {};

\draw [decorate,decoration={brace,amplitude=10pt},xshift=-4pt,yshift=0pt] (0,0.05) -- (0,5/2-.05) node [black,midway,xshift=-0.8cm] { $d_{1,m}^+$};
\draw [decorate,decoration={brace,amplitude=10pt},xshift=-4pt,yshift=0pt] (0,5/2+.05) -- (0,10/2-.05) node [black,midway,xshift=-0.8cm] { $d_{1,m}^-$};
\draw [decorate,decoration={brace,amplitude=10pt,mirror},yshift=-4pt,xshift=0pt] (0+.05,0) -- (6/2-.05,0) node [black,midway,yshift=-0.8cm] { $d_{2,m}^-$};
\draw [decorate,decoration={brace,amplitude=10pt,mirror},yshift=-4pt,xshift=0pt] (6/2+.05,0) -- (10/2-.05,0) node [black,midway,yshift=-0.8cm] { $d_{2,m}^+$};
\end{tikzpicture}
\caption{Illustration of a rectangle $R_m$ of the rectangular tessellation $\{R_m\}_{m=1}^s$, defined in Definition \ref{def.rectangular.tess}.}\label{fig.2}
\end{figure}

For a rectangular tessellation $\{ R_m \}_{m=1}^s$ we let
$d_m^{--} $ be the area of the rectangle $R_m$ North-West of $t_m$, $d_m^{+-}$ the area to the
South-West, $d_m^{++}$ the area to the South-East and $d_m^{-+}$ the area to the North-East.
In other words, if
$(t_{1,m}^-, t_{2,m}^-)$, $(t_{1,m}^-,  t_{2,m}^+)$, $(t_{1,m}^+, t_{2,m}^+)$, $(t_{1,m}^+,  t_{2,m}^-)$
are the four corners
of the rectangle $R_m$, starting with the top-left corner and going clockwise along the boundary,
then
\begin{align*}
d_m^{--} &= d_{1,m}^- d_{2,m}^- , & d_m^{-+} &= d_{1,m}^- d_{2,m}^+ , \\
d_m^{+-} &= d_{1,m}^+d_{2,m}^-  ,& d_m^{++} &= d_{1,m}^+ d_{2,m}^+ , 
\end{align*}
where
\begin{align*}
 d_{1,m}^- &= (t_{1,m} - t_{1,m}^- ) , & d_{2,m}^- & = (t_{2,m} - t_{2,m}^- ) ,\\
 d_{1,m}^+ & = (t_{1,m}^+ - t_{1,m}  ) , &  d_{2,m}^+& = (t_{2,m}^+ - t_{2,m}  ) .
\end{align*}
The rectangle $R_m$ is illustrated in Figure \ref{fig.2}.

 Fix  a set $S \subseteq [3 : n_1-1 ] \times [3: n_2-1]$ and a rectangular tessellation $\{R_m\}_{m=1}^s$. 
  Let $ d_{1, {\rm max} } (S) := \max_{m \in [1:s] } \max\{ d_{1,m}^-, d_{1,m}^+ \}$ and $ d_{2, {\rm max} } (S) := \max_{m \in [1:s] } \max\{ d_{2,m}^-, d_{2,m}^+ \}  $. The quantity $d_{1, {\rm max} } (S)$ (respectively  $ d_{2, {\rm max} } (S)$) denote the maximal horizontal (respectively vertical) distance from  a jump location to the boundary of the corresponding rectangular region in the rectangular tessellation $\{R_m\}_{m=1}^s$.

For simplicity, we do not elaborate on the dependence  on the sign configuration in the effective sparsity and focus instead on a worst-case upper bound holding for all sign configurations. In other words, we bound the worst case $\Gamma (S, {v}_{-S} ):=\max_{q_S} \Gamma (S, {v}_{-S}, q_S )$ rather than  $\Gamma (S, {v}_{-S}, q_S )$.

 \begin{theorem}[Adaptivity of image denoising with total variation] \label{main.theorem}  
 Let $g\in \R^{n_1 \times n_2}$ be arbitrary.  Let $x,\ t>0$.
Choose
 $$ \lambda \ge 2    \sqrt {\frac{d_{1, {\max}} (S)}{ n_1}+ \frac{d_{2, {\rm max} }(S) }{n_2}} \  \lambda_0 (t)  .
  $$
Then, with probability at least $1- e^{-x}-e^{-t}$, it holds that
\begin{equation*}
 \| \hat {\tilde f} - \tilde f^0  \|_2^2 /n \le\| { g} - \tilde f^0 \|_2^2 / n  + 4\lambda \norm{(\Delta { g})_{-S}}_1  +
\biggl (\sigma \sqrt { s\over n} + \sigma \sqrt {2x \over n } + {\lambda \Gamma (S, {v}_{-S} )}  \biggr )^2,    \end{equation*}
 where
 \begin{equation}\label{Gamma-bound.equation}
  \Gamma^2 (S , {v}_{-S} )  \le {1 \over 2} \biggl (\log  (e n_1)+ \log ( e n_2)  \biggr )
\sum_{m=1}^s \biggl ( {n \over d_m^{--} } +{n \over d_m^{-+ }}+ {n \over d_m^{++ }  } + {n\over d_m^{+-} } \biggr )     .
 \end{equation}
\end{theorem}

We note that the upper bound on the mean squared error in the above theorem depends on ${\tilde f}^0$, $g$ and $S$. The true underlying interaction terms ${\tilde f}^0$ are typically fixed, while the choices of $g$ and - under some constraints - of $S$ are arbitrary. As such, the upper bound can  be optimized over $g$ and $S$. A pair $(g=f^*,S=S^*)$ optimizing the upper bound is called ``oracle'' and depends on ${\tilde f}^0$. The optimized upper bound depends therefore on ${\tilde f}^0$ and consists of the optimal tradeoff between the approximation error of ${\tilde f}^0$ by the oracle signal $f^*$ and the estimation error of the piecewise rectangular structure encoded in the oracle active set $S^*$. This piecewise rectangular structure is estimated almost at a parametric rate, this means, almost as if the number and the locations of the elements of $S^*$ were known.

In this sense, Theorem \ref{main.theorem} is an adaptive result: the bound on the mean squared error of the estimator $\hat{\tilde f}$ of the interaction terms varies depending on the underlying interaction terms ${\tilde f}^0$ to estimate. The estimator $\hat{\tilde f}$ can therefore sense the structure in ${\tilde f}^0$ -- as for instance the number and the location of its jumps -- and adapt to it.

 Theorem \ref{main.theorem} gives also a theoretical justification for choosing the tuning parameter smaller than the universal choice $\lambda=\lambda_0(t)$. However the active set of the true image $f^0$ or of its oracle approximation might not be known in practice, so that one might have to choose $\lambda= \lambda_0(t)$. The choice $\lambda= \lambda_0(t)$ in the setting of Theorem \ref{thm4} with $s_1=s_2$ results then in an oracle bound of order $s^{2} /n $, up to log-terms.

Theorem \ref{main.theorem} can be obtained from Theorem  \ref{2d-appB-thm3} by finding a bound on the (worst-case) effective sparsity $\Gamma(S,v_{-S})$, which is the main contribution of Section \ref{TV2D-JMLR-Section7}.


In Subsection \ref{interpolating.section}, we expose some tools needed to bound the effective sparsity. Of the particular interest is the concept of interpolating matrix, which is an adaptation of the interpolating vector by \cite{vand19} to two dimensions. 

In Subsection \ref{effective-sparsity.section} we will take  the interpolating matrix for given and bound the effective sparsity based on it and on the tools exposed in Subsection \ref{interpolating.section}. 

The results of Subsection \ref{noise.section} will show that the interpolating matrix given in Subsection \ref{effective-sparsity.section} is indeed a valid interpolating matrix. 

Subsection \ref{proof.main.thm.section} combines the results of Subsections \ref{interpolating.section}-7.4 \ref{noise.section} to prove Theorem \ref{main.theorem}.
%

\subsection{Interpolating matrix and partial integration} \label{interpolating.section}

We now rewrite the definition of effective sparsity (Definition \ref{effective.sparsity}) in matrix instead of vector form.
Let $q_S=\{(q_S)_{j,k}\}_{(j,k)\in [2:n_1]\times [2:n_2]}$ and $v=\{v_{j,k}\}_{(j,k)\in [2:n_1]\times [2:n_2]}$ be a sign configuration and  noisy weights written in matrix form.

\begin{definition}[Effective sparsity in matrix form]\label{effective.sparsity.matrix}
The effective sparsity is defined as
$$\Gamma ( S, {v}_{-S},q_S ) = \max  \{   {\rm trace}(q_S^{T}  ( D_1 f D_2^T)_S)  - \| (1- {v})_{-S}  \odot ( D_1 f D_2^T)_{-S} \|_1  :\| f \|_2^2 / n=1 \}.$$
Moreover we write
$$ \Gamma ( S, {v}_{-S}):= \max_{q_S}\ \Gamma ( S, {v}_{-S},q_S ).$$
\end{definition}

We define an interpolating matrix. The interpolating matrix will be a tool for finding a bound on the effective sparsity. The concept of interpolating vector (and matrix) is inspired by the dual certificate by \cite{cand14}. Related concepts appear also earlier in the literature in \cite{fuch04b,cand13a}. \cite{vand19} make a connection between interpolating matrix and effective sparsity. This connection is here extended to the two-dimensional case. 

\begin{definition}[Interpolating matrix]
Let $q_S \in \{ -1,0,1 \}^{(n_1-1)\times (n_2-1)} $ be a sign configuration and $v \in [0,1]^{(n_1-1)\times (n_2-1)}$ be a matrix of weights.
We call  an interpolating matrix for the sign configuration $q_S$ and the weights $v$ a matrix $w (q_S)= \{ w_{j,k} (q_S)\}_{(j,k)\in [ 2 : n_1]\times [2: n_2]}$ having the following properties:\\
$\bullet$ $w_{t_m}(q_S) = q_{t_m} $, $\forall m\in [1:s] $,\\
$\bullet$ $|w_{j,k}(q_S)  | \le 1- {v}_{j,k}  , \ \forall (j,k) \notin S $.
\end{definition}

The interpolating matrix $w (q_S)$ can be interpreted to belong to the subdifferential of $\norm{\Delta h}_1$ for some matrix  $h\in \R^{n_1 \times n_2}$ with the same sign configuration. The choices of the active set $S$, the sign configuration $q_S$ and the matrix $h$ are not tied to the estimator $\hat{\tilde f}$.

For completeness, we give the matrix version of Lemma 4.2 in \cite{vand19}.

\begin{lemma}[How to bound the effective sparsity]\label{interpolating.lemma}
We have
$$\Gamma^2  (S, {v}_{-S},q_S ) \le n \min_{w(q_S)}  \| D_1^T w(q_S)   D_2 \|_2^2$$
where the minimum is over all interpolating matrices $w(q_S) $ for the sign configuration $q_S$.  
\end{lemma}

\begin{proof}
See Appendix \ref{interpolating.proof}.
\end{proof}

The proof of Lemma \ref{interpolating.lemma} uses the equation
$$ {\rm trace} (w^T  D_1 f D_2^T ) = {\rm trace}( D_2^T w^T D_1 f ) .$$
When $D_1 f D_2^T = \Delta f $ this equality is called
partial integration. We  study it further in the next lemma.

\begin{lemma}[Partial integration in two dimensions with zero boundaries]\label{partial-integration.lemma}
Choose arbitrarily a matrix $w =
\{ w_{j,k}\}_{(j,k) \in [2: n_1 ] \times [2: n_2] } $ with its boundary entries equal to zero, i.e., 
$$ w_{j, 2}= w_{j, n_2} =0, \ \forall \ j \in [2: n_1] \text{ and } w_{2, k}=w_{n_1, k} =0, \ \forall \ k \in [2 : n_2] . $$
Then 
 $${\rm trace}(w^T\Delta f) = \sum_{k=2}^{n_2}  \sum_{j=2}^{n_1} w_{j,k} (\Delta f)_{j,k} =
 \sum_{j=2}^{n_1 -1} \sum_{k=2}^{n_2-1}(\Delta w)_{j+1 , k+1}  f_{j,k}  . $$
\end{lemma}

\begin{proof}
See Appendix \ref{part.int.proof}.
\end{proof}

To obtain a bound on the effective sparsity, we now have to find a suitable interpolating matrix.  We then compute the Frobenius norm of its total derivative with the help of partial integration.

\subsection{A bound for the effective sparsity} \label{effective-sparsity.section} 
Given a set $S \subseteq [3, n_1-1] \times [3, n_2-1]$, let  $\{ R_m \}_{m=1}^s $ be a rectangular tessellation. 
Each jump location $t_m$ is an interior point of $R_m$.
The rectangle $R_m$ consists of a North-West rectangle $R_m^{--}$ a
North-East rectangle $R_m^{-+}$ a South-East rectangle $R^{++}$ and a South-West
rectangle $R^{+-}$.  Thus
\begin{eqnarray*}
R_m^{--}&:=&  \{ (j,k): \ t_{1,m}^-  \le j \le t_{1,m}  , \ t_{2,m}^-   \le k \le t_{2,m}\}, \\
 R_m^{-+}&:= & \{ (j,k): \ t_{1,m}^-  \le j \le t_{1,m} , \ t_{2, m}  \le k \le t_{2, m}^+  \} ,\\
R_m^{++}&:= & \{ (j,k): \ t_{1,m}  \le j \le t_{1,m}^+, \ t_{2, m}  \le k \le t_{2, m}^+ \}, \\
R_m^{+-}&:= & \{ (j,k): \ t_{1,m}  \le j \le t_{1,m}^+ , \ t_{2,m}^-   \le k \le t_{2,m} \} . 
\end{eqnarray*}

Let $z_1,\ z_2 \in \{+,-\}$. For $m=1 , \ldots , s$ and $(j,k)\in R_m^{z_1 z_2}$  the weights will be
\begin{equation}\label{weights.equation}
 {v}_{j,k} =  1- {1 \over 2} \left (1- \sqrt {\frac{|j-t_{1,m} |}{ d_{1,m}^{z_1}}} \right )\left  (1- {\frac{|k-t_{2,m} |}{ d_{2,m}^{z_2} }}\right ) -  {1 \over 2} \left (1- {\frac{|j-t_{1,m} |}{d_{1,m}^{z_1}}} \right )\left  (1- \sqrt {\frac{|k-t_{2,m}|}{d_{2,m}^{z_2}} }\right ) ,\end{equation} 
where $( d_{1,m}^{-} , d_{1,m}^+, d_{2,m}^-, d_{2,m}^+)$ are given in Subsection \ref{main.result.section}. 

For $q_S\in \{ - 1,0,1\}^{(n_1-1)\times (n_2-1)}  $, take 
\begin{equation}\label{q.equation}
w_{j,k} (q_S) =\begin{cases}  +1- {v}_{j,k}, & q_{t_m}=+1  \cr
 -1 + {v}_{j,k} , & q_{t_m} =-1 \cr \end{cases}, \ (j,k) \in R_m, \ m \in [s].
\end{equation}
Then $w(q_S)$ is an interpolating matrix for $q_S$. Moreover, it has the property that
$w_{j,k} (z)=0$ as soon as $(j,k) $ is at the boundary of $R_m$ for some $m \in [s]$. 

\begin{remark}[Dependence on the rectangular tessellation]
Given $S$, the weights  in \eqref{weights.equation} and the interpolating matrix  in \eqref{q.equation} both depend on the rectangular tessellation $\{R_m\}_{m \in [s]}$ chosen. Thus, also the bound on the effective sparsity derived in the next lemma depends on $\{R_m\}_{m \in [s]}$ and can be interpreted to hold, given a set $S$,  for an arbitrary rectangular tessellation $\{R_m\}_{m \in [s]}$.
\end{remark}

\begin{lemma}[Bound on the (worst-case) effective sparsity]\label{Gamma.lemma}
With the weights ${v}$ given in (\ref{weights.equation}) we have
$$\Gamma^2 (S, {v}_{-S}) \le\frac{1}{2}\biggl ( \log (en_1)+ \log (en_2)\biggr ) 
\sum_{m=1}^s \biggl ( { n\over d_{m}^{-- } }+{n \over d_{m}^{-+}  }+ {n \over d_m^{++}  } + {n \over d_m^{+-}}\biggr )     .$$
\end{lemma}
 
\begin{proof}
 We say that the interpolating matrix $w$ has product structure if it is of the form $w (j,k) = w_1(j) w_2(k)$ for all $
(j,k)\in [2:n_1]\times [2:n_2]$. Clearly, if it has this structure, then 
$$ ( \Delta w)_{j,k} = (D_1w_1)_j (D_2 w_2 )_k . $$

We examine now a prototype rectangle
$[-d_1^-:d_1^+]\times [-d_2^-:d_2^+] $. 
Consider the four rectangles
\begin{align*}
R^{--}&:=  [-d_1^-:0]\times [-d_2^-:0], & R^{-+}&:=  [-d_1^-:0]\times [0:d_2^+], \\
R^{+-}&:=  [0:d_1^+]\times [-d_2^-:0] , & R^{++}&:=  [0:d_1^+]\times [0:d_2^+] , \\
\end{align*}
and let $R := R^{--} \cup R^{-+} \cup R^{+-} \cup R^{++} =[-d_1^-:d_1^+]\times [-d_2^-:d_2^+] $. Thus $R$ is a rectangle surrounding the origin $(0,0)$. The prototype rectangle $R$ is illustrated in Figure \ref{fig.3}.

\begin{figure}[h]
\centering
\begin{tikzpicture}
\draw (0,0) rectangle (10/2,10/2);
\draw (0,5/2) -- (10/2,5/2);
\draw (6/2,0) -- (6/2,10/2);

\draw[pattern=north west lines] (0,0) rectangle (6/2,5/2);
\draw[pattern=north west lines] (6/2,5/2) rectangle (10/2,10/2);
\draw[pattern=north east lines] (0,5/2) rectangle (6/2,10/2);
\draw[pattern=north east lines] (6/2,0) rectangle (10/2,5/2);

\node[fill=white,rounded corners=2pt] at (2.8/2,2.4/2) {$R^{+,-}$};
\node[fill=white,rounded corners=2pt] at (2.8/2,7.4/2) {$R^{-,-}$};
\node[fill=white,rounded corners=2pt] at (8/2,7.4/2) {$R^{-,+}$};
\node[fill=white,rounded corners=2pt] at (8/2,2.4/2) {$R^{+,+}$};

\node[circle,inner sep=2pt,fill=black,label={[shift={(-.85,-.7)}]$(d_{1}^+,-d_{2}^-)$}] at (0,0) {};
\node[circle,inner sep=2pt,fill=black,label={[shift={(.75,-.7)}]$(d_{1}^+,d_{2}^+)$}] at (10/2,0) {};
\node[circle,inner sep=2pt,fill=black,label={[shift={(0.9,-.2)}]$(-d_{1}^-,d_{2}^+)$}] at (10/2,10/2) {};
\node[circle,inner sep=2pt,fill=black,label={[shift={(-1,-.2)}]$(-d_{1}^-,-d_{2}^-)$}] at (0,10/2) {};

\node[circle,inner sep=2pt,fill=black,label={[shift={(.8/2,0)},fill=white]$(0,0)$}] at (6/2,5/2) {};

\draw [decorate,decoration={brace,amplitude=10pt},xshift=-4pt,yshift=0pt] (0,0.05) -- (0,5/2-.05) node [black,midway,xshift=-0.8cm] { $d_{1}^+$};
\draw [decorate,decoration={brace,amplitude=10pt},xshift=-4pt,yshift=0pt] (0,5/2+.05) -- (0,10/2-.05) node [black,midway,xshift=-0.8cm] { $d_{1}^-$};
\draw [decorate,decoration={brace,amplitude=10pt,mirror},yshift=-4pt,xshift=0pt] (0+.05,0) -- (6/2-.05,0) node [black,midway,yshift=-0.8cm] { $d_{2}^-$};
\draw [decorate,decoration={brace,amplitude=10pt,mirror},yshift=-4pt,xshift=0pt] (6/2+.05,0) -- (10/2-.05,0) node [black,midway,yshift=-0.8cm] { $d_{2}^+$};
\end{tikzpicture}
\caption{Illustration of the prototype rectangle $R$ used in the Proof of Lemma \ref{Gamma.lemma}.}\label{fig.3}
\end{figure}
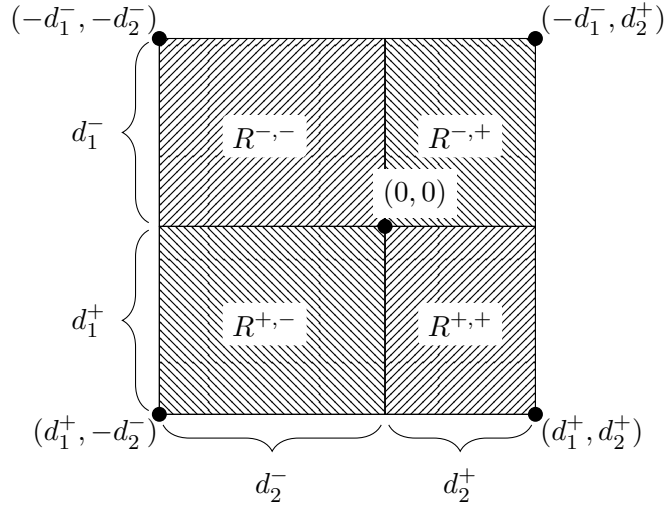

For $z_1, \ z_2 \in \{+,-\}$, and $(j,k) \in R^{z_1 z_2}$, take
\begin{equation*}
 { w}_{j,k} : = {1 \over 2} \left (1- \sqrt {|j| \over d_1^{z_1}} \right )\left  (1-  {|k|\over  d_2^{z_2} }\right ) +  {1 \over 2} \left (1-  {|j|\over  d_1^{z_1}} \right )\left  (1- \sqrt {|k|\over  d_2^{z_2} }\right ).
\end{equation*}
Then ${w}_{0,0} =1$ and ${ w}_{j,k} =0 $ for all $(j,k) $ at the border of $R$.

Because $R^{--} $, $R^{-+}$, $R^{++} $  and $R^{+-}$ are rectangles aligned
with the coordinate axes,  $w $ is the sum of two terms with product structure. 
We see that
$$| \Delta { w}_{j,k} | \le  {1 \over 2} {1 \over \sqrt { d_1^{z_1} |j| } } {1 \over d_2^{z_2}}  +
{1 \over 2} {1 \over { d_1^{z_1}  } } {1 \over \sqrt { d_2^{z_2} |k| }},   (j,k)\in R^{z_1 z_2}.
$$
Invoking the inequality $(a+b)^2 \le 2a^2 + 2 b^2 $ for real numbers $a$ and $b$, we conclude that
\begin{eqnarray*}
\sum_{(j,k) \in R} \biggl ( \Delta { w}_{j,k} \biggr )^2 & \le & 
\sum_{(z_1,z_2)\in \{+,-\}^2}\left( {1 \over 2} { 1 \over d_1^{z_1} (d_2^{z_2})^2 } \sum_{j=1}^{d_1^{z_1} } \sum_{k=1}^{d_2^{z_2}} { 1 \over j} + 
{1 \over 2} { 1 \over (d_1^{z_1} )^2 d_2^{z_2} } \sum_{j=1}^{d_1^{z_1} } \sum_{k=1}^{d_2^{z_2}} { 1 \over k} \right) 
 \\ 
&\le & {1 \over 2} { 1 \over d_1^- d_2^- } \left( \log (e d_1^-) + \log (e d_2^+) \right )+ {1 \over 2} { 1 \over d_1^- d_2^+ } \left( \log (e d_1^-) + \log (e d_2^-) \right ) \\
&&+  {1 \over 2} { 1 \over d_1^+ d_2^+ } \left ( \log (e d_1^+) + \log (e d_2^+) \right ) +{1 \over 2} { 1 \over d_1^+ d_2^- } \left ( \log (e d_1^+) + \log (e d_2^-) \right ).
\end{eqnarray*}

The interpolating matrices $w(q_S)$ given by (\ref{q.equation}) are of the above form 
on each   $R_m$. Moreover, they are equal to
zero on their borders. 
The final result follows from glueing the $\{ R_m \}_{m=1}^s$ together.
\end{proof}

\subsection {Dealing with the noise}\label{noise.section}

We start with an auxiliary lemma, which will be  used to find a convenient formula for the interpolating matrix $w$ based on the noise weights $v$. Both $w$ and $v$ are informally added to the statement of the lemma under the terms to which they  correspond in its application.

\begin{lemma}[Auxiliary lemma] \label{auxiliary.lemma} For all $(x,y) \in [0,1]^2 $
$$\underbrace{ \frac{(1- \sqrt x) (1- y) +  ( 1- x) (1-\sqrt y)}{2} )}_{``w"}  \le 1 - \underbrace{  \frac{\sqrt x+ \sqrt  y}{2} }_{``{v}"} $$
\end{lemma}

\begin{proof}
See Appendix \ref{auxiliary.lemma.proof}.
\end{proof}

The next lemma will be used to obtain the inverse scaling factor ${\tilde \gamma}$ and the noise weights $v$ from the bound on the antiprojections ${\tilde v}$. Here too ${\tilde v}$, $v$ and ${\tilde \gamma}$ are added below the terms to which they correspond in the application of the lemma.

\begin{lemma}[Finding noise weights] \label{boundweights.lemma} For any $((t_1, t_2), (d_1, d_2)) \in \mathbb{N}^4 $, for $j \in [ t _1:t_1+d_1]$ and $k \in  [t_2 : t_2+ d_2]$,
$$ \underbrace{\sqrt {{j-t_1 \over n_1} +{ k-t_2\over n_2}}}_{``{\tilde v}"} \le 
 \underbrace{\biggl ( \sqrt { {j-t_1 \over d_1}} + \sqrt {{k-t_2 \over d_2} }  \biggr )}_{``2v"} \underbrace{\sqrt{{ d_1 \over n_1} + {d_2 \over n_2} }}_{``{\tilde \gamma}/2"} .$$
\end{lemma} 

\begin{proof}
See Appendix \ref{boundweights.lemma.proof}.
\end{proof}


 
By Lemma \ref{2d-appB-l02}, we  can  bound the antiprojections using the distance of the inactive variables $\{\tilde \psi^{j,k} \}_{(j,k) \notin S } $
from the linear space spanned by the active ones ($\{ \tilde \psi^{j,k} \}_{(j,k) \in S } $).
As a consequence
of the next lemma, we may also
look at the original variables ${\psi^{j,k}}_{(j,k)\in [n_1]\times [n_2]}$ instead.

\begin{lemma}[Projections] \label{projections.lemma} Consider the linear spaces ${\cal U} = {\rm span}(\{ u_j\} ) $ and ${\cal W}$.
Define the linear space $\tilde {\cal U}$ as $\tilde {\cal U} :={\rm span} ( \{ {\tilde u_j}  \} )$, where $\tilde u_j = {\rm P}_{\cal W} u_j$ for all
$j$.
For any $z$ define $\tilde z = {\rm P}_{\cal W} z $.
Then $\|  \tilde z - {\rm P}_{\tilde {\cal U} } \tilde z \|_2  \le \| z - {\rm P}_{\cal U} z \|_2 $.

\end{lemma}

\begin{proof}
See Appendix \ref{projections.lemma.proof}
\end{proof}

In the next lemma we  bound the antiprojections using the distance of the original inactive variables $\{ \psi^{j,k} \}_{(j,k) \notin S } $
from the linear space spanned by the active ones ($\{ \psi^{j,k} \}_{(j,k) \in S } $).

Let $ {\cal U} := {\rm span} \left (  \{  \psi^{t_m} \}_{m=1}^s \right )$.

\begin{lemma}[Finding a bound on the antiprojections]\label{antiprojections.lemma} For $m \in [s] $ and all $(j,k) \in R_m$

$$ \|  {\rm A}_{\cal U} \psi^{j,k}  \|_2^2 / n \le \frac{\abs{j-t_{1,m}}}{n_1}+ \frac{\abs{k-t_{2,m}}}{n_2}.$$
\end{lemma}

\begin{proof}
See Appendix \ref{antiprojections.lemma.proof}.
\end{proof}

\subsection{Proof of Theorem \ref{main.theorem}}\label{proof.main.thm.section}

Theorem \ref{main.theorem} follows from the results of Subsections \ref{interpolating.section}-\ref{noise.section} combined with Theorem \ref{2d-appB-thm3}.

We thus need to find suitable $\tilde{v}, \ \tilde{\gamma}, \ v, \ w$ and an upper bound on the effective sparsity.

Let $S\subseteq [3:n_1-1]\times [3:n_2-1]$ be arbitrary.
Let
$$ \mathcal{U}:= \text{span}\left(\{\psi^{t_m}\}_{m \in [s]} \right), \ \tilde{\mathcal{U}}:= \text{span}\left(\{\tilde{\psi}^{t_m}\}_{m \in [s]} \right)$$
 and
 $$ \mathcal{W}:= \text{span}\left( \{\psi^{j,k}\}_{(j,k)\in \{1\}\times[n_2] \cup [n_1]\times \{1\}} \right).$$
 Note that $\norm{\text{A}_{\tilde{\mathcal{U}}}\tilde{\psi}^{j,k}}_2/ \sqrt{n}$ is the same quantity as $\norm{({\rm I}_n-{\rm P}_{S}) {\tilde \Psi}_i}^2_2/ \sqrt{n}$ found in Definition \ref{def.bound.antiproj}.

\begin{itemize}
\item By Lemma \ref{projections.lemma} we have that $ \norm{\text{A}_{\tilde{\mathcal{U}}}\tilde{\psi}^{j,k}}_2/ \sqrt{n}=\norm{\text{A}_{\tilde{\mathcal{U}}}\tilde{\psi}^{j,k}}_2/ \sqrt{n} \le \norm{\text{A}_{{\mathcal{U}}}{\psi}^{j,k}}_2/ \sqrt{n}$, since $\tilde{\psi}^{j,k}= \text{P}_{\mathcal{W}} \psi^{j,k}$, $(j,k) \in [2:n_1] \times [2:n_2]$.
\item Upper bounds on the values of $\norm{\text{A}_{{\mathcal{U}}}{\psi}^{j,k}}_2/ \sqrt{n}$ are given by Lemma \ref{antiprojections.lemma}. These upper bounds yield a bound on the antiprojections $\tilde v$ and the corresponding inverse scaling factor $\tilde \gamma$.
\item By applying Lemma \ref{boundweights.lemma} and Lemma \ref{auxiliary.lemma} to the results of Lemma \ref{antiprojections.lemma} we see that for $v$ as in \eqref{weights.equation} we have $ v_{j,k} \ge \tilde{v}_{j,k} / \tilde{\gamma}, \ \forall (j,k) \in [2:n_1]\times [2:n_2]$, where 
$$ \tilde{\gamma}= 2 \sqrt{\frac{d_{1, \max}(S)}{n_1}+\frac{d_{2, \max}(S)}{n_2}}$$
and $v$ reaches its maximal values of $1$ on the boundaries of the rectangles $\{R_m\}_{m=1}^s$.
\item Therefore $w(q_S)$ given in \eqref{q.equation} is an interpolating matrix for the sign configuration $q_S$.
\end{itemize}

Lemma \ref{Gamma.lemma} bounds the effective sparsity by using the interpolating matrix $w(q_S)$ given in \eqref{q.equation}. \hfill  $\blacksquare$
\section{Slow rates for image denoising}\label{TV2D-JMLR-Section8}

The objective for this section is to prove a slow rate for $\hat{\tilde{f}}$. It will turn out that this rate is $n^{-5/8}$, up to log terms, and is faster than the rate $n^{-3/5}$  in \cite{mamm97-2}.

We combine the standard Theorem \ref{2d-appB-thm4} with the insight that, if the arbitrary active set $S$ is chosen in a careful way, we can obtain a value of ${\tilde \gamma}$ which is smaller than the one obtained in Section \ref{TV2D-JMLR-Section7}. The key idea is that an active set defining a ``mesh grid'' (see Definition \ref{def.mesh.grid} and Figure \ref{fig.mesh.grid}) results in a more favorable ${\tilde \gamma}$ than an active set $S$ defining a regular grid.

We consider the two-dimensional grid $[2:n_1] \times [2:n_2]$.
Let $t_1, t_2 \in \mathbb{N}$. Assume that $ {n_1}/{(t_1^2+1)}$ and ${n_2}/{(t_2^2+1)}$ are integers.

We define
\begin{eqnarray*}
M_1&:=& \left\{1+\frac{n_1}{t_1^2+1}, 1+\frac{2n_1}{t_1^2+1}, \ldots, 1+\frac{t^2_1n_1}{t_1^2+1}  \right\}\\
N_1&:=& \left\{ 1+\frac{\lceil t_1/2 \rceil n_1}{t_1^2+1},  1+\frac{(\lceil t_1/2 \rceil+t_1)n_1}{t_1^2+1}, \ldots, 1+\frac{(\lceil t_1/2 \rceil+t_1(t_1-1))n_1}{t_1^2+1} \right\}\\
M_2 &:=& \left\{ 1+\frac{n_2}{t_2^2+1}, 1+\frac{2n_2}{t_2^2+1}, \ldots, 1+\frac{t_2^2n_2}{t_2^2+1} \right\}  \\
N_2 &:=&  \left\{ 1+\frac{\lceil t_2/2 \rceil n_2}{t_2^2+1},  1+\frac{(\lceil t_2/2 \rceil+t_2)n_2}{t_2^2+1}, \ldots, 1+\frac{(\lceil t_2/2 \rceil+t_2(t_2-1))n_2}{t_2^2+1} \right\}
\end{eqnarray*}

\begin{definition}[Mesh grid]\label{def.mesh.grid}
An active set $S_M \subseteq [2:n_1] \times [2:n_2]$ is said to define a mesh grid if $ S_M:= M_1 \times N_2 \cup N_1 \times M_2$.

We define the set of the nodes $S_N$ of the mesh grid $S_M$ as $ S_N:= N_1 \times N_2$.
\end{definition}

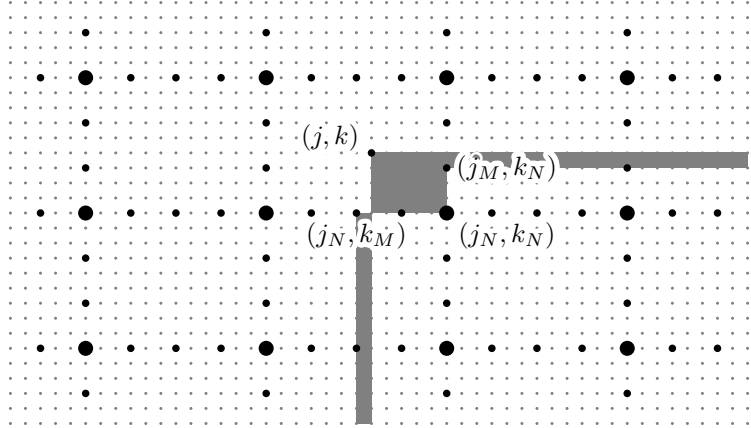
\begin{figure}[h]
\centering
\begin{tikzpicture}

\fill[gray]  (4.6,0) rectangle (4.8,2.8);
\fill[gray]  (5.8,3.4) rectangle (9.8,3.6);
\fill[gray]  (4.8,02.8) rectangle (5.8,3.6);

    \foreach \x in {0,.2,...,9.8} {
        \foreach \y in {0,.2,..., 5.6} {
            \fill[color=gray] (\x,\y) circle (0.02);
        }
    }
\foreach \x in {1,3.4,5.8,8.2} {
        \foreach \y in {0.4, 1.0, 1.6, 2.2, 2.8, 3.4, 4.0, 4.6, 5.2} {
            \fill[color=black] (\x,\y) circle (0.05);
        }
    }
\foreach \x in {0.4, 1.0, 1.6, 2.2, 2.8, 3.4, 4.0, 4.6, 5.2, 5.8, 6.4, 7.0, 7.6, 8.2, 8.8, 9.4} {
        \foreach \y in {1.0,2.8,4.6} {
            \fill[color=black] (\x,\y) circle (0.05);
        }
    }
    
\foreach \x in {1,3.4,5.8,8.2} {
        \foreach \y in {1.0,2.8,4.6} {
            \fill[color=black] (\x,\y) circle (0.1);
        }
    }
\fill[color=black] (4.8,3.6) circle (0.05);
\node at (4.25,3.8) {\contour{white}{\small $(j,k)$}};

\node at (4.6,2.5) {\contour{white}{\small $(j_N,k_M)$}};
\node at (6.6,3.4) {\contour{white}{\small $(j_M,k_N)$}};
\node at (6.6,2.5) {\contour{white}{\small $(j_N,k_N)$}};

\end{tikzpicture}
\caption{Illustration of the two-dimensional grid $[2:30] \times [2:51]$. The black points represent a mesh grid $S_M$ with $t_1=3$ and $t_2=4$. The thicker points belong to the set of nodes $S_N$. The grey area corresponds to the error made when approximating $\psi^{j,k}$ by a linear combination of $\psi^{j_N,k_N}, \psi^{j_N,k_M}, \psi^{j_M,k_N}\in \mathcal{M}$  in the Proof of Lemma \ref{mesh.antiproj}.}\label{fig.mesh.grid}
\end{figure}

An example of a mesh grid is illustrated in Figure \ref{fig.mesh.grid}. If $S_M$ is a mesh grid, then $ s_M=t_1^2t_2 + t_1 t_2^2 -t_1t_2= t_1t_2(t_1+t_2-1)$.
Note that we can add $t_1t_2$ points to the mesh grid $S_M$ to obtain an active set of cardinality $t_1t_2(t_1+t_2)$. The antiprojections do not become larger. With some abuse of notation we write from now on $s_M= t_1t_2(t_1+t_2)$.

The next lemma is an analogon of Lemma \ref{antiprojections.lemma}, but the active set is constrained to take the form of a mesh grid. This allows us to find a more favorable uniform bound for the antiprojections, which is smaller than the one we would have if the active set were a regular grid like $S_N$ instead of a mesh grid.

Let $S_M$ be a mesh grid and let $\mathcal{M}:= {\rm span}\left(\{\psi^{j,k}\}_{(j,k)\in S_M} \right)$.

\begin{lemma}[Finding a bound on the antiprojections when $S$ is a mesh grid]\label{mesh.antiproj}
For all $(j,k) \in [2:n_1] \times [2:n_2]$ it holds that
$$ \norm{{\rm A}_{\mathcal{M}}\psi^{j,k}}^2_2/n \le \frac{1}{t_1^2} + \frac{1}{t_2^2} + \frac{\lceil t_1/2 \rceil \lceil t_2/2 \rceil}{t_1^2 t_2^2}.$$
\end{lemma}

\begin{proof} 
Let $(j,k) \in [2:n_1] \times [2:n_2]$ be arbitrary. We define
\begin{align*}
j_M&:= \arg \min_{i \in M_1} \abs{i-j} & j_N&:= \arg\min_{i \in N_1} \abs{i-j}\\
k_M&:= \arg \min_{i \in M_2} \abs{i-k} & k_N&:= \arg\min_{i \in N_2} \abs{i-k}.\\
\end{align*}
We now approximate $\psi^{j,k}$ by a linear combination of $\psi^{j_N,k_N}, \psi^{j_N,k_M}, \psi^{j_M,k_N}\in \mathcal{M}$ (cf. Figure \ref{fig.mesh.grid})  as
\begin{eqnarray*}
\norm{\psi^{j,k}- \psi^{j_M,k_N}- \psi^{j_N,k_M}+ \psi^{j_N,k_N}}^2_2& \le & j \abs{k-k_M} + k \abs{j-j_M} + \abs{j_N-j_M} \abs{k_N-k_M}\\
& \le & \frac{n_1n_2}{t_2^2+1} + \frac{n_1n_2}{t_1^2+1} + \frac{n_1n_2\lceil t_1/2 \rceil \lceil t_2/2 \rceil}{(t_1^2+1)(t_2^2+1)}.
\end{eqnarray*}
\end{proof}

Let us now define the class
$$ \mathcal{F}(C)= \left\{ {\tilde f} \in \R^{n_1 \times n_2}: {\rm TV}({\tilde f}) \le C \right\}, \ C>0.$$

For the ease of exposition, we consider square images ($n_1=n_2$) and we choose $t_1=t_2=: t$ to be even. The following theorem holds.

\begin{theorem}[Slow rates for total variation image denoising]\label{main.slow}
The estimator $\hat{\tilde f} $ has the following properties.
\begin{itemize}
\item {\bf Dependence of $S_M$ and $\lambda$ on $f^0$ allowed.} \\
Choose $S_M$ such that
$$ s_M=\left\lceil\frac{2^{5/4}3^{3/4}\log^{3/8}(2n){\rm TV}^{3/4}(f^0)}{\sigma^{3/4}}  \right\rceil \text{ and } \lambda = \frac{3^{3/4}\sigma^{5/4}\log^{3/8}(2n)}{2^{1/12}n^{5/8}{\rm TV}^{1/4}(f^0)}\ge {\tilde \gamma}\lambda_0(\log(2n)). $$
Then, with probability at least $1-1/n$ it holds that
$$ \norm{\hat{\tilde f}-{\tilde f}^0}^2_2/n \le \frac{24 \sigma^{5/4}C^{3/4}\log^{3/8}(2n)}{n^{5/8}} + \frac{4 \sigma^2\log(2n)}{n}+ \frac{2\sigma^2}{n}.$$

\item {\bf Dependence of $S_M$ and $\lambda$ on $f^0$ not allowed.} \\
Choose $S_M$ such that
$$ s_M=\left\lceil2^{5/4}3^{3/4}\log^{3/8}(2n)  \right\rceil \text{ and } \lambda = \frac{3^{3/4}\log^{3/8}(2n)}{2^{1/12}n^{5/8}} \ge {\tilde \gamma}\lambda_0(\log(2n)). $$
Then, with probability at least $1-1/n$ it holds that
$$ \norm{\hat{\tilde f}-{\tilde f}^0}^2_2/n \le \frac{12 \sigma(C+\sigma) \log^{3/8}(2n)}{n^{5/8}} + \frac{4 \sigma^2\log(2n)}{n}+ \frac{2\sigma^2}{n}.$$

\end{itemize}
\end{theorem}

\begin{proof}
Our starting point is Theorem \ref{2d-appB-thm4}. We choose $x=t=\log(2n)$, $g={\tilde f}^0$ and $S=S_M$ and we apply the inequality $(\sqrt{2x}+\sqrt{s_M})^2 \le 4x + 2s_M$. Note that $s_M=2t^3$. By Lemma \ref{mesh.antiproj}, ${\tilde \gamma}=3t/2=2^{-2/3}3 s_M^{-1/3}$.

The two results of the above theorem follow by optimally choosing $s_M$ to trade off $2s_M/n$ and $4 \lambda{\rm TV}(f^0)  $ or $4 \lambda/\sigma$, respectively.
\end{proof}

\begin{remark}[Comments on Theorem \ref{main.slow}]
\begin{itemize}
\item Theorem \ref{main.slow} improves on the rate $n^{-3/5}$ found in \cite{mamm97-2}. If in the proof of Theorem \ref{main.slow} one chooses an active set defining a regular grid and bounds the antiprojections with Lemma \ref{antiprojections.lemma}, then the rate $n^{-3/5}$ by \cite{mamm97-2} is retrieved, up to a $log^{2/5}n$ term.
\item Since we assume $t$ to be an even integer, $s_M$ should be the smallest number, which is twice an even cube and greater than the $s_M$ we propose in the statement of Theorem \ref{main.slow}.
\item In the second part of Theorem \ref{main.slow} the choice of $\lambda$ is completely data-driven, does not depend on ${\tilde f}^0$ and is smaller than the universal choice $\lambda_0(\log(2n))$.
\item Under the assumption that $C>0$ is a constant that does not depend on $n$, the two rates in  Theorem \ref{main.slow} are equal.
\item The results of Theorem \ref{main.slow} are ``constant-friendly'' in the sense that we can trace the constants for the choice of the tuning parameter and for the upper bound on the mean squared error. Moreover, these constants are small. This situation has to be contrasted with results relying on  entropy calculations (for instance \cite{mamm97-2}), which  possibly produce very large constants both in the tuning parameter and in the upper bound.
\end{itemize}
\end{remark}

%

\begin{remark}[Comparision with \cite{fang19}]
\cite{fang19} study the estimator
$$\hat{f}_C:= \arg \min_{f \in K(C)} \norm{Y-f}^2_n,$$
where, for $C>0$, $K(C):=\{f \in \R^{n_1 \times n_2}:  \sum_{j=2}^{n_1} \abs{f_{j,1}-f_{j-1,1}} + \sum_{k=2}^{n_2} \abs{f_{1,k}-f_{1,k-1}} + {\rm TV}(f)\le C\}$. \cite{fang19} show that the minimax rate of estimation for $f^0\in K(C)$ is of order $n^{-2/3}C^{2/3}$ up to log terms and $\hat{f}_C$ attains it. 

With the entropy bound by \cite{blei07} used  in \cite{fang19}, one could prove a similar result for the penalized version of the estimator $\hat{f}_C$, which differs from  our estimator $\hat{f}$, since the one-dimensional total variation is defined in a different way and the tuning parameters are choosen to be $\lambda_1=\lambda_2= \lambda$.

The minimax rate in the class $\bar{K}(C):=  \{ f \in \R^{n_1\times n_2}: {\rm TV}_1(f)\le C, {\rm TV}_2(f)\le C  , {\rm TV}(f)\le C  \}$ follows, by the ANOVA decomposition, from the minimax rate for the class $\mathcal{F}(C)=\{\tilde{f} \in \R^{n_1 \times n_2}: {\rm TV}(\tilde{f})\le C\}$. This minimax rate or a tight bound for the entropy of $\mathcal{F}(C)$ are not clear yet. The class $\bar{K}(C)$ is larger than the class $K(C)$, as for $f \in K(C)$ it holds that ${\rm TV}_1(f)\le C$ and ${\rm TV}_2(f)\le C$.
\end{remark}
\section{Conclusion}\label{TV2D-JMLR-Section9}

We  showed that the estimator for the interaction terms satisfies an oracle inequality with fast rates and can adapt to the number and the locations of the unknown ``jumps'' to optimally trade off approximation and estimation error as if it would know the true image $f^0$.

As in our previous work on one-dimensional (higher-order) total variation regularization, the projection arguments by \cite{dala17} are central to make sure that the effective sparsity is ``small enough'' and thus prove adaptivity. This technique exploits the strong correlation between the atoms constituting the dictionary. It also turns out that the technique to bound the effective sparsity proposed by \cite{vand19} and  inspired by \cite{cand14} can be generalized to the two-dimensional case.

The slow rate $n^{-5/8}\log^{3/8}n$ improves on the rate $n^{-3/5}$ in \cite{mamm97-2}.

Both for fast and slow rates,  the estimator enjoys the most favorable theoretical properties when the tuning parameter is chosen to be smaller than the universal choice of order $\sqrt{\log n/n}$. The most favorable choice of the tuning parameter might depend on some aspects of the true image $f^0$. However we show that there are choices of $\lambda$ which are completely data driven and still confer to the estimator adaptivity (fast rates) or the rate $n^{-5/8}\log^{3/8}n$ (slow rates).


\acks{We would like to acknowledge support for this project
from the the Swiss National Science Foundation (SNF grant 200020\_169011). We moreover thank the associate editor and the referees for the careful reading of the manuscript and for their valuable comments.}


\newpage

\appendix

\section{Proofs of Section \ref{TV2D-JMLR-Section5}}\label{TV2D-JMLR-Section10}

\subsection{Proof of Lemma \ref{expansion.lemma}}\label{proof.expansion.lemma}
We note that, for $(j,k)\in [n_1]\times [n_2]$, the global mean of dictionary atoms is given by $\psi^{j,k}(\circ, \circ)= \psi^{1,k}(\circ,\circ)\psi^{j,1}(\circ,\circ)$, the main effects by $\psi^{j,k}(\cdot, \circ)= \psi^{1,k}(\circ, \circ) \psi^{j,1}-\psi^{j,k}(\circ,\circ) \psi^{1,1}$ and $\psi^{j,k}(\circ, \cdot)= \psi^{j,1}(\circ, \circ) \psi^{1,k}-\psi^{j,k}(\circ,\circ) \psi^{1,1}$.
 Thus
 $$ \psi^{j,k}= \tilde \psi^{j,k}+ \psi^{1,k}(\circ, \circ) \tilde \psi^{j,1}+ \psi^{j,1}(\circ, \circ)\tilde \psi^{1,k} + \psi^{j,k}(\circ,\circ) \psi^{1,1}.$$
 From the definitions of $\tilde \psi^{j,k},\ (j,k)\in [n_1]\times [n_2]$ it follows that
 
 \begin{equation*}
 f= \tilde \beta_{1,1} \tilde \psi^{1,1} + \sum_{j=2}^{n_1} \tilde \beta_{j,1} \tilde \psi^{j,1}+ \sum_{k=2}^{n_2} \tilde \beta_{1,k} \tilde \psi^{1,k} + \sum_{j=2}^{n_1} \sum_{k=2}^{n_2} \tilde \beta_{j,k} \tilde \psi^{j,k},
 \end{equation*}
 where
$$
 \tilde \beta_{j,k}=\begin{dcases} \beta_{1,1} + \sum_{j=2}^{n_1} \beta_{j,1} \psi^{j,1}( \circ, \circ) + \sum_{k=2}^{n_2} \beta_{1,k} \psi^{1,k}(\circ, \circ)+ \sum_{j=2}^{n_1} \sum_{k=2}^{n_2} \beta_{j,k} \psi^{j,k}(\circ, \circ),\  (j,k)=(1,1)\\
\beta_{j,1} + \sum_{k=2}^{n_2} \beta_{j,k} \psi^{1,k}(\circ, \circ),\ (j,k) \in [2:n_1]\times [1],\\
\beta_{1,k} + \sum_{j=2}^{n_1} \beta_{j,k} \psi^{j,1}(\circ, \circ), \ (j,k) \in [1]\times [2: n_2],\\
 \beta_{j,k}, \ (j,k) \in [2:n_1]\times [2:n_2].
\end{dcases}$$
 Note that $\psi^{1,k}(\circ, \circ)= 1-(k-1)/n_2$, $\psi^{j,1}(\circ, \circ)= 1-(j-1)/n_1$ and 
 $$ \sum_{j=2}^{n_1} \beta_{j,1} \psi^{j,1}(\circ, \circ) = -f(1,1) + \frac{1}{n_1} \sum_{j=1}^{n_1} f(j,1).$$
 Analogously, it holds that
 $$ \sum_{k=2}^{n_2} \beta_{j,k} \psi^{1,k}(\circ, \circ)= -f(1,1)+ \frac{1}{n_2} \sum_{k=1}^{n_2} f(1,k).$$
 By plugging in  the expressions for $\beta_{j,k}, \ (j,k) \in [n_1]\times [n_2] $ and by using the above equations the result follows. 
\hfill $\blacksquare$

\subsection{Proof of Lemma \ref{separate-interactions.lemma}}\label{proof.separate-interactions.lemma}
Since $f- \tilde f $ and $\tilde f$ are orthogonal, ${\rm trace} ( \tilde f^T (f- \tilde f) )=0$.
We have
\begin{eqnarray*}
 \| Y - f \|_2^2 - \| Y \|_2^2 &=& -2 {\rm trace} (Y^T  f) + \| f \|_2^2\\
 & =& -2 {\rm trace} (Y^T (f -\tilde f)) -2 {\rm trace} (Y^T \tilde f ) + \| f - \tilde f \|_2^2 + \| \tilde f \|_2^2 \\
&= & \| Y - ( f - \tilde f ) \|_2^2 + \| Y - \tilde f \|_2^2- 2\| Y \|_2^2 . 
\end{eqnarray*}
The result now follows from Lemma \ref{expansion.lemma}.
\hfill $\blacksquare$

\section{Proofs of Section \ref{TV2D-JMLR-Section6}}\label{TV2D-JMLR-Section11}

For the proof of Lemma \ref{2d-appB-l01} we use the dual norm inequality and the polarization identity.

Let $y,z \in \R^n$. The dual norm of $\norm{z}_1$ is $ \sup_{\norm{y}_1\le 1} \abs{y^Tz}= \max_{i \in [n]} \abs{z_i}= \norm{z}_{\infty}$. In particular, the dual norm inequality holds
$$ \abs{y^Tz} \le \norm{y}_1 \norm{z}_{\infty}, \ \forall y, z \in \R^n.$$

For $y,z \in \R^n$ the polarization identity  holds 
$$ 2 y^Tz= \norm{y}^2_2 + \norm{z}^2_2 - \norm{y-z}^2_2.$$

\subsection{Proof of Lemma \ref{2d-appB-l01}}\label{proof.basic.inequ}

The estimator $\hat{\tilde f}$ satisfies the optimality conditions (KKT conditions) 
$$ \frac{{\tilde Y}-\hat{\tilde f}}{n}= \lambda \Delta^{T} \partial \norm{\Delta \hat{\tilde f}}_1,$$
where $\Delta'\partial\norm{\Delta\hat{\tilde f}}_1$ is a subgradient of $\norm{\Delta\hat{\tilde f}}_1$ by the chain rule of the subgradient (Theorem 23.9 in \cite{rock70}). In other words $\Delta'\partial\norm{\Delta\hat{\tilde f}}_1$ is any vector s.t. $\hat{\tilde f}' \Delta' \partial \norm{\Delta\hat{\tilde f}}_1 = \norm{\Delta\hat{\tilde f}}_1$. Since $\Delta$ is fixed, $\partial \norm{\Delta\hat{\tilde f}}_1$ can be any vector in $\R^{(n_1-1)(n_2-1)}$ s.t. 
$$
(\partial \norm{\Delta\hat{\tilde f}}_1)_i\in \begin{cases} \sign((\Delta\hat{\tilde f})_i), & (\Delta \hat{\tilde f})_i \not= 0, \\
[-1,1], &  (\Delta \hat{\tilde f})_i = 0.
\end{cases}
$$
The set of subgradients of $\norm{\Delta\hat{\tilde f}}_1$ is called subdifferential of $\norm{\Delta\hat{\tilde f}}_1$.
By multiplying the KKT conditions by $\hat{\tilde f}$ and an arbitrary $g \in \R^n$, we therefore obtain
\begin{equation}\label{eq.basic.1} \frac{\hat{\tilde f}^{T}({\tilde Y}-\hat{\tilde f})}{n}= \lambda \hat{\tilde f}^{T} \Delta^{T} \partial \norm{\Delta \hat{\tilde f}}_1= \lambda \norm{\Delta \hat{\tilde f}}_1\end{equation}
and
\begin{equation}\label{eq.basic.2} \frac{g^{T}({\tilde Y}-\hat{\tilde f})}{n}= \lambda g^{T} \Delta^{T} \partial \norm{\Delta \hat{\tilde f}}_1\le \lambda \norm{\Delta g}_1.\end{equation}
The last inequality follows by the dual norm inequality and the fact that $\norm{\partial \norm{\Delta \hat{\tilde f}}_1}_{\infty} \le 1$.
Subtracting Equation \eqref{eq.basic.1} from Equation \eqref{eq.basic.2}, together with the identity ${\tilde Y}= {\tilde f}^0+ {\tilde \epsilon}$ yields
$$ \frac{(g-\hat{\tilde f})^{T}({\tilde Y}-\hat{\tilde f})}{n}=\frac{(g-\hat{\tilde f})^{T}({\tilde f}^0-\hat{\tilde f})}{n} - \frac{{\tilde \epsilon}^T(\hat{\tilde f}-g)}{n} \le \lambda (\norm{\Delta g}_1- \norm{\Delta \hat{\tilde f}}_1).$$
The application of the polarization identity gives
$$2 \frac{(g-\hat{\tilde f})^{T}({\tilde f}^0-\hat{\tilde f})}{n}= \norm{\hat{\tilde f}-{\tilde f}^0}^2_2/n + \norm{\hat{\tilde f}-g}^2_2/n - \norm{g-{\tilde f}^0}^2_2/n$$
and Lemma \ref{2d-appB-l01} follows.

\subsection{Proof of Lemma \ref{2d-appB-l02}}\label{bound.ep.proof}

We start by decomposing the empirical process as in Equation \eqref{Eq.decompose}. For the first part, we have that
$$ {{\tilde \epsilon}^{T} {\rm P}_{S} {\tilde f}}/{n} \le \norm{ {\rm P}_{S} {\tilde \epsilon}}_2 \norm{{\tilde f}}_2  / n \le \norm{ {\rm P}_{S} {\tilde \epsilon}}_2 \norm{f}_2  / n$$
and by applying Lemma 1 by \cite{laur00} or Lemma 8.6 by \cite{vand16}   for $x>0$ it holds that with probability at least $1-e^{-x}$
$$ {{\tilde \epsilon}^{T} {\rm P}_{S} {\tilde f}}/{n} \le {\norm{f}_2}\sigma (\sqrt{2x} + \sqrt{s} )/n.  $$

For the second part, note that  by Lemma  \ref{expansion.lemma} we can write
$$ {{\tilde \epsilon}^{T} {\rm A}_{S} {\tilde f} }/{n}= {{\tilde \epsilon}^{T} {\rm A}_{S} {\tilde \Psi} {\tilde \beta} }/{n},$$
where $ {\tilde \beta}= \Delta f$.

 Then by Lemma 17.5 in \cite{vand16}, for $t>0$ we have that  with probability at least $1-e^{-t}$
$$ \frac{{\tilde \epsilon}^{T} {\rm A}_{S} {\tilde f}}{n} \le \lambda \norm{\tilde{v}_{-S}\odot ( \Delta {\tilde f})_{-S}}_1/\tilde{\gamma}= \lambda \norm{{v}_{-S} \odot (\Delta {f})_{-S}}_1, $$
if we choose $\lambda \ge {\tilde \gamma} \lambda_0(t)$.
Thus the claim follows.
\hfill $\blacksquare$

\subsection{Proof of Theorem \ref{2d-appB-thm3}}\label{fast.rate.proof}

For $S \subseteq [(n_1-1)(n_2-1)]$ and $q_S= \text{sign}((\Delta g)_{S})$,  we have that, by the triangle inequality,
\begin{equation*}
\norm{\Delta g}_1 - \norm{\Delta \hat{\tilde f}}_1  \le  q_S^{T}(\Delta(g- \hat{\tilde f}))_{S} -  \norm{(\Delta(g- \hat{\tilde f}))_{-S}}_1+ 2 \norm{(\Delta g)_{-S}}_1.
\end{equation*}
By combining the above inequality with Lemma \ref{2d-appB-l01} and Lemma \ref{2d-appB-l02} applied to $g- \hat{\tilde f}$, using the definition of effective sparsity, and applying the convex conjugate inequality to the term involving $\norm{g- \hat{\tilde f}}_2/\sqrt{n}$ we obtain the claim.
\hfill $\blacksquare$

\subsection{Proof of Theorem \ref{2d-appB-thm4}}\label{slow.rate.proof}
Note that Lemma \ref{2d-appB-l02} implies that, for $x,t>0$ and $\lambda \ge {\tilde \gamma} \lambda_0(t)$, with probability at least $1-e^{-x}-e^{-t}$ it holds that
\begin{equation*}
2\frac{{\tilde \epsilon}^{T} (g-\hat{\tilde f})}{n} \le \frac{\norm{g-\hat{\tilde f}}^2_2}{{n}}+ \left(\sqrt{\frac{2x}{n}} + \sqrt{\frac{s}{n}} \right)^2+ 2 \lambda (\norm{\Delta g}_1 + \norm{\Delta \hat{\tilde f}}_1).
\end{equation*}
Combining the above claim with the basic inequality (Lemma \ref{2d-appB-l01}) proves the theorem.
\hfill $\blacksquare$

\section{Proofs of Section \ref{TV2D-JMLR-Section7}}\label{TV2D-JMLR-Section12}

\subsection{Proof of Lemma \ref{partial-integration.lemma}}\label{part.int.proof}
Partial integration in one dimension gives, for  $N \in \mathbb{N}$ and sequences
$ \{ a_j \}_{j=2}^N : a_2=a_N=0$ and $\{ b_j \}_{j=1}^N $ of real numbers,
$$ \sum_{j=2}^N a_j ( b_j - b_{j-1}) = a_N b_N - a_2 b_1 - \sum_{j=2}^{N-1} (a_{j+1} - a_j) b_j=  - \sum_{j=2}^{N-1} (a_{j+1} - a_j) b_j.$$
Apply this to the inner and the outer sum.
\hfill $\blacksquare$

\subsection{Proof of Lemma \ref{interpolating.lemma}}\label{interpolating.proof}
Let $f\in \R^{n_1 \times n_2} $ be arbitrary and let $q_S$ be a sign configuration.
Then
\begin{eqnarray*}
 && {\rm trace}(q_S^{T} \odot ( D_1 f D_2^T)_S)  - \| (1- {v})_{-S}  \odot ( D_1  f D_2^T)_{-S} \|_1\\
  &\le&  {\rm trace}(q_S^{T} \odot ( D_1 f D_2^T)_S) - \norm{w_{-S}(q_{S})  \odot ( D_1  f D_2^T)_{-S}}_1\\
 & =&  {\rm trace}(w(q_S)^T D_1 f D_2^T) =  {\rm trace}(D_2^Tw(q_S)^T D_1 f)\\
 & \le&  \sqrt{n} \norm{D_1^T w(q_S) D_2}_2 \norm{f}_2/\sqrt{n}.
\end{eqnarray*} 
\hfill $\blacksquare$

\subsection{Proof of Lemma \ref{auxiliary.lemma}}\label{auxiliary.lemma.proof}
By direct calculation
\begin{eqnarray*}
 1- {1 \over 2}(1- \sqrt x) (1- y) - {1 \over 2} ( 1- x) (1-\sqrt y) 
 &=&
{1 \over 2} \sqrt x + {1 \over 2} (1-\sqrt x)y + {1 \over 2} \sqrt y + { 1\over 2} x (1- \sqrt y) \\
 &\ge & (\sqrt x + \sqrt y)/ 2
 \end{eqnarray*} 
 where we used that $(x,y) \in [0,1]^2 $ so that $(1-\sqrt x)y$ and $x (1- \sqrt y)$ are non-negative. 
 \hfill $\blacksquare$
 
\subsection{Proof of Lemma \ref{boundweights.lemma}}\label{boundweights.lemma.proof}
For $j \in \{ t _1, \ldots, t_1+d_1 \}$ and $k \in \{ t_2 , \ldots , t_2+ d_2 \}$
\begin{eqnarray*}
 \sqrt {{j-t_1 \over n_1} +{ k-t_2\over n_2}}\le \sqrt {j-t_1 \over n_1 } + \sqrt {k-t_2\over n_2} &=&
\sqrt {j-t\over d_1} \sqrt {d_1\over n_1} + \sqrt {k-t_2\over d_2 } \sqrt {d_2 \over n_2 } \\
& \le&  \left(\sqrt { {j-t_1 \over d_1}} + \sqrt{{k-t_2 \over d_2} }\right) \sqrt {{ d_1 \over n_1} + {d_2 \over n_2} } .
\end{eqnarray*}
 \hfill $\blacksquare$
 
 \subsection{Proof of Lemma \ref{projections.lemma}}\label{projections.lemma.proof}
 Clearly $ \| {\rm P}_{\cal W}(z-  {\rm P}_{\cal U} z)\|_2 \le \| z- {\rm P}_{\cal U} z \|_2$. We moreover have  for some vector $\gamma$
$${\rm P}_{\cal U}  z = \sum_{j} \gamma_j u_j  \text{ s.t. }{\rm P}_{\cal W}(z-  {\rm P}_{\cal U} z) = \tilde z - \sum_{j} \gamma_j \tilde u_j  . $$ Thus
\begin{equation*}
 \|  \tilde z - {\rm P}_{\tilde U} \tilde z \|_2  =  \min_{c} \| \tilde z - \sum_j c_j \tilde u_j \|_2 \le \| \tilde z - \sum_{j} \gamma_j \tilde u_j \|_2= \| {\rm P}_{\cal W} (z- {\rm P}_{\cal U} z ) \|_2  \le \| z - {\rm P}_{\cal U} z \|_2 .
\end{equation*} 
 \hfill $\blacksquare$
 
 \subsection{Proof of Lemma \ref{antiprojections.lemma}}\label{antiprojections.lemma.proof}
 For $(j,k) \in R_m$,
$$ \| {\rm A}_{\cal U} \psi^{j,k} \|^2_2  \le \| \psi^{j,k} - \psi^{t_m} \|^2_2 \le  
{\abs{j- t_{1,m}}n_2 }  +  {\abs{k- t_{2,m}} n_1} . $$
  \hfill $\blacksquare$
  
%




\vskip 0.2in
\bibliography{library,/Users/fortelli/PhD/newlibrary/library}

\begin{thebibliography}{58}
\providecommand{\natexlab}[1]{#1}
\providecommand{\url}[1]{\texttt{#1}}
\expandafter\ifx\csname urlstyle\endcsname\relax
  \providecommand{\doi}[1]{doi: #1}\else
  \providecommand{\doi}{doi: \begingroup \urlstyle{rm}\Url}\fi

\bibitem[Abergel and Moisan(2017)]{aber17}
R.~Abergel and L.~Moisan.
\newblock {The Shannon total variation}.
\newblock \emph{Journal Math Imaging Vis}, 59:\penalty0 341--370, 2017.

\bibitem[Arias-Castro et~al.(2012)Arias-Castro, Salmon, and Willett]{aria12}
E.~Arias-Castro, J.~Salmon, and R.~Willett.
\newblock {Oracle inequalities and minimax rates for nonlocal means and related
  adaptive kernel-based methods}.
\newblock \emph{SIAM Journal on Imaging Sciences}, 5\penalty0 (3):\penalty0
  944--992, 2012.

\bibitem[Bach(2011)]{bach11b}
F.~Bach.
\newblock {Shaping level sets with submodular functions}.
\newblock \emph{Neural Information Processing Systems (NIPS)}, pages 10--18,
  2011.

\bibitem[Bellec(2018)]{bell18}
P.~Bellec.
\newblock {Sharp oracle inequalities for least-squares estimators in shape
  restricted regression}.
\newblock \emph{Annals of Statistics}, 46\penalty0 (2):\penalty0 745--780,
  2018.

\bibitem[Bellec et~al.(2017)Bellec, Salmon, and Vaiter]{bell17}
P.~Bellec, J.~Salmon, and S.~Vaiter.
\newblock {A sharp oracle inequality for graph-slope}.
\newblock \emph{Electronic Journal of Statistics}, 11\penalty0 (2):\penalty0
  4851--4870, 2017.

\bibitem[Bellec et~al.(2018)Bellec, Lecu{\'{e}}, and Tsybakov]{bell16}
P.~Bellec, G.~Lecu{\'{e}}, and A.~Tsybakov.
\newblock {Slope meets Lasso: improved oracle bounds and optimality}.
\newblock \emph{Annals of Statistics}, 46\penalty0 (6B):\penalty0 3603--3642,
  2018.

\bibitem[Belloni et~al.(2011)Belloni, Chernozhukov, and Wang]{bell11}
A.~Belloni, V.~Chernozhukov, and L.~Wang.
\newblock {Square-root Lasso: pivotal recovery of sparse signals via conic
  programming}.
\newblock \emph{Biometrika}, 98\penalty0 (4):\penalty0 791--806, 2011.

\bibitem[Blei et~al.(2007)Blei, Gao, and Li]{blei07}
R.~Blei, F.~Gao, and W.~Li.
\newblock {Metric entropy of high dimensional distributions}.
\newblock \emph{Proceedings of the American Mathematical Society}, 135\penalty0
  (12):\penalty0 4009--4018, 2007.

\bibitem[Cand{\`{e}}s and Fernandez-Granda(2014)]{cand14}
E.~Cand{\`{e}}s and C.~Fernandez-Granda.
\newblock {Towards a mathematical theory of super-resolution}.
\newblock \emph{Communications on Pure and Applied Mathematics}, 67\penalty0
  (6):\penalty0 906--956, 2014.

\bibitem[Cand{\`{e}}s and Recht(2013)]{cand13a}
E.~Cand{\`{e}}s and B.~Recht.
\newblock {Simple bounds for recovering low-complexity models}.
\newblock \emph{Math. Program. Ser. A}, 141:\penalty0 577--589, 2013.

\bibitem[Caselles et~al.(2015)Caselles, Chambolle, and Novaga]{case15}
V.~Caselles, A.~Chambolle, and M.~Novaga.
\newblock {Total variation in imaging}.
\newblock \emph{Handbook of Mathematical Methods in Imaging: Volume 1, Second
  Edition}, 1\penalty0 (1):\penalty0 1455--1499, 2015.

\bibitem[Chambolle and Lions(1997)]{cham97}
A.~Chambolle and P.~L. Lions.
\newblock {Image recovery via total variation minimization and related
  problems}.
\newblock \emph{Numerische Mathematik}, 76\penalty0 (2):\penalty0 167--188,
  1997.

\bibitem[Chambolle et~al.(2010)Chambolle, Caselles, Cremers, and
  Novaga]{cham10}
A.~Chambolle, V.~Caselles, D.~Cremers, and M.~Novaga.
\newblock {An introduction to total variation for image analysis}.
\newblock In \emph{Radon Series Comp. Appl. Math. 9}. De Gruyter, 2010.

\bibitem[Chambolle et~al.(2011)Chambolle, Levine, and Lucier]{cham11}
A.~Chambolle, S.~E. Levine, and B.~J. Lucier.
\newblock {An upwind finite-difference method for total variation-based image
  smoothing}.
\newblock \emph{Siam J. Imaging Sciences}, 4\penalty0 (1):\penalty0 277--299,
  2011.

\bibitem[Chambolle et~al.(2017)Chambolle, Duval, Peyr{\'{e}}, and Poon]{cham17}
A.~Chambolle, V.~Duval, G.~Peyr{\'{e}}, and C.~Poon.
\newblock {Geometric properties of solutions to the total variation denoising
  problem}.
\newblock \emph{Inverse Problems}, 33\penalty0 (1), 2017.

\bibitem[Chatterjee and Goswami(2019)]{chat19}
S.~Chatterjee and S.~Goswami.
\newblock {New risk bounds for 2d total variation denoising}.
\newblock \emph{arXiv:1902.01215v2}, 2019.

\bibitem[Condat(2017)]{cond17}
L.~Condat.
\newblock {Discrete total variation: New definition and minimization}.
\newblock \emph{SIAM Journal on Imaging Sciences}, 10\penalty0 (3):\penalty0
  1258--1290, 2017.

\bibitem[Dabov et~al.(2007)Dabov, Foi, and Egiazarian]{dabo07}
K.~Dabov, A.~Foi, and K.~Egiazarian.
\newblock {Video denoising by sparse 3D transform-domain collaborative
  filtering}.
\newblock \emph{IEEE Transactions on Control of Network Systems}, 16\penalty0
  (8):\penalty0 2080--2095, 2007.

\bibitem[Dalalyan and Salmon(2012)]{dala12a}
A.~Dalalyan and J.~Salmon.
\newblock {Sharp oracle inequalities for aggregation of affine estimators}.
\newblock \emph{Annals of Statistics}, 40\penalty0 (4):\penalty0 2327--2355,
  2012.

\bibitem[Dalalyan et~al.(2017)Dalalyan, Hebiri, and Lederer]{dala17}
A.~Dalalyan, M.~Hebiri, and J.~Lederer.
\newblock {On the prediction performance of the Lasso}.
\newblock \emph{Bernoulli}, 23\penalty0 (1):\penalty0 552--581, 2017.

\bibitem[Elad(2010)]{elad10}
M.~Elad.
\newblock \emph{{Sparse and Redundant Representations}}.
\newblock Springer, 2010.

\bibitem[Elad et~al.(2007)Elad, Milanfar, and Rubinstein]{elad07}
M.~Elad, P.~Milanfar, and R.~Rubinstein.
\newblock {Analysis versus synthesis in signal priors}.
\newblock \emph{Inverse Problems}, 23\penalty0 (947), 2007.

\bibitem[Elsener and van~de Geer(2019)]{else19}
A.~Elsener and S.~van~de Geer.
\newblock {Sharp oracle inequalities for stationary points of nonconvex
  penalized M-estimators}.
\newblock \emph{IEEE Transactions on Information Theory}, 65\penalty0
  (3):\penalty0 1452--1472, 2019.

\bibitem[Fadili and Peyre(2011)]{fadi11}
J.~Fadili and G.~Peyre.
\newblock {Total variation projection with first order schemes}.
\newblock \emph{IEEE Transactions on Image Processing}, 20\penalty0
  (3):\penalty0 657--669, 2011.

\bibitem[Fang et~al.(2019)Fang, Guntuboyina, and Sen]{fang19}
B.~Fang, A.~Guntuboyina, and B.~Sen.
\newblock {Multivariate extensions of isotonic regression and total variation
  denoising via entire monotonicity and Hardy-Krause variation}.
\newblock \emph{arXiv:1903.01395v1}, 2019.

\bibitem[Fuchs(2004)]{fuch04b}
J.-J. Fuchs.
\newblock {On sparse representations in arbitrary redundant bases}.
\newblock \emph{IEEE Transactions on Information Theory}, 50\penalty0
  (6):\penalty0 1341--1344, 2004.

\bibitem[Goyal et~al.(2020)Goyal, Dogra, Agrawal, Sohi, and Sharma]{goya20}
B.~Goyal, A.~Dogra, S.~Agrawal, B.~S. Sohi, and A.~Sharma.
\newblock {Image denoising review: from classical to state-of-the-art
  approaches}.
\newblock \emph{Information Fusion}, 55:\penalty0 220--244, 2020.

\bibitem[Guntuboyina et~al.(2020)Guntuboyina, Lieu, Chatterjee, and
  Sen]{gunt20}
A.~Guntuboyina, D.~Lieu, S.~Chatterjee, and B.~Sen.
\newblock {Adaptive risk bounds in univariate total variation denoising and
  trend filtering}.
\newblock \emph{Annals of Statistics}, 48\penalty0 (1):\penalty0 205--229,
  2020.

\bibitem[H{\"{u}}tter and Rigollet(2016)]{hutt16}
J.-C. H{\"{u}}tter and P.~Rigollet.
\newblock {Optimal rates for total variation denoising}.
\newblock \emph{JMLR: Workshop and Conference Proceedings}, 49:\penalty0 1--32,
  2016.

\bibitem[Koltchinskii(2006)]{kolt06}
V.~Koltchinskii.
\newblock {Local Rademacher complexities and oracle inequalities in risk
  minimization}.
\newblock \emph{Annals of Statistics}, 34\penalty0 (6):\penalty0 2593--2656,
  2006.

\bibitem[Laurent and Massart(2000)]{laur00}
B.~Laurent and P.~Massart.
\newblock {Adaptive estimation of a quadratic functional by model selection}.
\newblock \emph{Annals of Statistics}, 28\penalty0 (5):\penalty0 1302--1338,
  2000.

\bibitem[Lin et~al.(2017)Lin, Sharpnack, Rinaldo, and Tibshirani]{lin17b}
K.~Lin, J.~Sharpnack, A.~Rinaldo, and R.~J. Tibshirani.
\newblock {A sharp error analysis for the fused Lasso, with application to
  approximate changepoint screening}.
\newblock \emph{Neural Information Processing Systems (NIPS)}, \penalty0
  (3):\penalty0 42, 2017.

\bibitem[Lounici et~al.(2011)Lounici, Pontil, van~de Geer, and
  Tsybakov]{loun11}
K.~Lounici, M.~Pontil, S.~van~de Geer, and A.~Tsybakov.
\newblock {Oracle inequalities and optimal inference under group sparsity}.
\newblock \emph{Annals of Statistics}, 39\penalty0 (4):\penalty0 2164--2204,
  2011.

\bibitem[Mammen and Tsybakov(1995)]{mamm95}
E.~Mammen and A.~Tsybakov.
\newblock {Asymptotical minimax recovery of sets with smooth boundaries}.
\newblock \emph{Annals of Statistics}, 23\penalty0 (2):\penalty0 502--524,
  1995.

\bibitem[Mammen and van~de Geer(1997)]{mamm97-2}
E.~Mammen and S.~van~de Geer.
\newblock {Locally adaptive regression splines}.
\newblock \emph{Annals of Statistics}, 25\penalty0 (1):\penalty0 387--413,
  1997.

\bibitem[Ortelli and van~de Geer(2018)]{orte18}
F.~Ortelli and S.~van~de Geer.
\newblock {On the total variation regularized estimator over a class of tree
  graphs}.
\newblock \emph{Electronic Journal of Statistics}, 12:\penalty0 4517--4570,
  2018.

\bibitem[Ortelli and van~de Geer(2019{\natexlab{a}})]{orte19-1}
F.~Ortelli and S.~van~de Geer.
\newblock {Synthesis and analysis in total variation regularization}.
\newblock \emph{ArXiv ID 1901.06418v1}, 2019{\natexlab{a}}.

\bibitem[Ortelli and van~de Geer(2019{\natexlab{b}})]{vand19}
F.~Ortelli and S.~van~de Geer.
\newblock {Prediction bounds for (higher order) total variation regularized
  least squares}.
\newblock \emph{ArXiv ID 1904.10871}, 2019{\natexlab{b}}.

\bibitem[Ortelli and van~de Geer(2020)]{orte19-2}
F.~Ortelli and S.~van~de Geer.
\newblock {Oracle inequalities for square root analysis estimators with
  application to total variation penalties}.
\newblock \emph{Information and Inference: A Journal of the IMA}, \penalty0
  (iaaa002), 2020.

\bibitem[Padilla et~al.(2018)Padilla, Scott, Sharpnack, and Tibshirani]{padi17}
O.~H.~M. Padilla, J.~Scott, J.~Sharpnack, and R.~Tibshirani.
\newblock {The DFS fused Lasso: linear-time denoising over general graphs}.
\newblock \emph{Journal of Machine Learning Research}, 18:\penalty0 1--36,
  2018.

\bibitem[Polzehl and Spokoiny(2003)]{polz03}
J.~Polzehl and V.~Spokoiny.
\newblock {Image denoising: pointwise adaptive approach}.
\newblock \emph{Annals of Statistics}, 31\penalty0 (1):\penalty0 30--57, 2003.

\bibitem[Rockafellar(1970)]{rock70}
R.~T. Rockafellar.
\newblock \emph{{Convex analysis}}, volume 196.
\newblock Princeton university press, 1970.

\bibitem[Rudin et~al.(1992)Rudin, Osher, and Fatemi]{rudi92}
L.~Rudin, S.~Osher, and E.~Fatemi.
\newblock {Nonlinear total variation based noise removal algorithms}.
\newblock \emph{Physica D}, 60:\penalty0 259--268, 1992.

\bibitem[Sadhanala et~al.(2016)Sadhanala, Wang, and Tibshirani]{sadh16}
V.~Sadhanala, Y.-X. Wang, and R.~Tibshirani.
\newblock {Total variation classes beyond 1d: minimax rates, and the
  limitations of linear smoothers}.
\newblock \emph{Neural Information Processing Systems (NIPS)}, 2016.

\bibitem[Sadhanala et~al.(2017)Sadhanala, Wang, Sharpnack, and
  Tibshirani]{sadh17b}
V.~Sadhanala, Y.~X. Wang, J.~Sharpnack, and R.~Tibshirani.
\newblock {Higher-order total variation classes on grids: Minimax theory and
  trend filtering methods}.
\newblock In \emph{Advances in Neural Information Processing Systems}, pages
  5801--5811, 2017.

\bibitem[Sharpnack et~al.(2012)Sharpnack, Rinaldo, and Singh]{shar12}
J.~Sharpnack, A.~Rinaldo, and A.~Singh.
\newblock {Sparsistency of the edge Lasso over graphs}.
\newblock \emph{International Conference on Artificial Intelligence and
  Statistics (AISTATS)}, 22:\penalty0 1028--1036, 2012.

\bibitem[Stucky and van~de Geer(2017)]{stuc17}
B.~Stucky and S.~van~de Geer.
\newblock {Sharp oracle inequalities for square root regularization}.
\newblock \emph{Journal of Machine Learning Research}, 18:\penalty0 1--29,
  2017.

\bibitem[Tibshirani(1996)]{tibs96}
R.~Tibshirani.
\newblock {Regression Shrinkage and Selection via the Lasso}.
\newblock \emph{J. R. Statist. Soc. B}, 58\penalty0 (1):\penalty0 267--288,
  1996.

\bibitem[Tibshirani(2014)]{tibs14}
R.~Tibshirani.
\newblock {Adaptive piecewise polynomial estimation via trend filtering}.
\newblock \emph{Annals of Statistics}, 42\penalty0 (1):\penalty0 285--323,
  2014.

\bibitem[Tibshirani et~al.(2005)Tibshirani, Saunders, Rosset, Zhu, and
  Knight]{tibs05}
R.~Tibshirani, M.~Saunders, S.~Rosset, J.~Zhu, and K.~Knight.
\newblock {Sparsity and smoothness via the fused Lasso}.
\newblock \emph{Journal of the Royal Statistical Society: Series B (Statistical
  Methodology)}, 67\penalty0 (1):\penalty0 91--108, 2005.

\bibitem[van~de Geer(2007)]{vand07b}
S.~van~de Geer.
\newblock {Oracle Inequalities and Regularization}.
\newblock In \emph{Lectures on Empirical Processes}, pages 191--252. European
  Mathematical Society, 2007.

\bibitem[van~de Geer(2009)]{vand09a}
S.~van~de Geer.
\newblock \emph{{Empirical Processes in M-estimation}}, volume~6 of
  \emph{Cambridge series in statistical and probabilistic mathematics}.
\newblock Cambridge University Press, Cambridge, 2009.

\bibitem[van~de Geer(2016)]{vand16}
S.~van~de Geer.
\newblock \emph{{Estimation and Testing under Sparsity}}, volume 2159.
\newblock Springer, 2016.

\bibitem[van~de Geer(2018)]{vand18}
S.~van~de Geer.
\newblock {On tight bounds for the Lasso}.
\newblock \emph{Jorunal of Machine Learning Research}, 19:\penalty0 1--48,
  2018.

\bibitem[van~de Geer and B{\"{u}}hlmann(2009)]{vand09b}
S.~van~de Geer and P.~B{\"{u}}hlmann.
\newblock {On the conditions used to prove oracle results for the Lasso}.
\newblock \emph{Electronic Journal of Statistics}, 3:\penalty0 1360--1392,
  2009.

\bibitem[Waghmare()]{wagh}
V.~E. Waghmare.
\newblock {Leaf Shapes Database}.
\newblock URL
  \url{http://imageprocessingplace.net/downloads_V3/root_downloads/image_databases/leaf
  shape database/leaf_shapes_downloads.htm}.

\bibitem[Wang et~al.(2016)Wang, Sharpnack, Smola, and Tibshirani]{wang16}
Y.-X. Wang, J.~Sharpnack, A.~Smola, and R.~Tibshirani.
\newblock {Trend filtering on graphs}.
\newblock \emph{Journal of Machine Learning Research}, 17:\penalty0 15--147,
  2016.

\bibitem[Zhang et~al.(2018)Zhang, Zuo, and Zhang]{zhan18}
K.~Zhang, W.~Zuo, and L.~Zhang.
\newblock {FFDNet: toward a fast and flexible solution for CNN-based image
  denoising}.
\newblock \emph{IEEE Transactions on Image Processing}, 27\penalty0
  (9):\penalty0 4608--4622, 2018.

\end{thebibliography}

\end{document}